\documentclass[11pt,a4paper]{article}
\usepackage{amsmath}
\usepackage[a4paper]{geometry}
\usepackage{amssymb,latexsym,amsmath,amsfonts,amsthm}
\usepackage{graphicx}
\usepackage{algorithm}
\usepackage{algorithmic}
\usepackage{dsfont}
\usepackage{mathrsfs}
\usepackage{epsfig}
\usepackage{booktabs}
\usepackage[toc,page]{appendix}
\usepackage{comment,verbatim}
\usepackage{overpic}
\usepackage{bbm}

\newtheorem{theorem}{Theorem}[section]
\newtheorem{lemma}[theorem]{Lemma}

\newtheorem{corollary}[theorem]{Corollary}

\theoremstyle{definition}

\theoremstyle{definition}
\newtheorem{example}[theorem]{Example}

\theoremstyle{remark}

\newtheorem{remark}[theorem]{Remark}

\numberwithin{equation}{section}

\usepackage{color}

\usepackage[normalem]{ulem}

\def\XXint#1#2#3{{
\setbox0=\hbox{$#1{#2#3}{\int}$}
\vcenter{\hbox{$#2#3$}}\kern-.5\wd0}}

\begin{document}
\title{Fast and highly accurate computation of Chebyshev expansion coefficients of analytic functions}
\author{Haiyong Wang\footnotemark[1]~ and Daan Huybrechs\footnotemark[2]
\\[3\jot]
KU Leuven \\
Department of Computer Science\\
Celestijnenlaan 200A, B-3001 Leuven, Belgium}

\maketitle
\renewcommand{\thefootnote}{\fnsymbol{footnote}}
\footnotetext[1]{School of Mathematics and Statistics, Huazhong
University of Science and Technology, Wuhan 430074, P. R. China.
E-mail: \texttt{haiyongwang@hust.edu.cn}  } \footnotetext[2]{Dept.
of Computer Science, KU Leuven, Belgium. E-mail:
\texttt{daan.huybrechs@cs.kuleuven.be}}

\begin{abstract}
Chebyshev expansion coefficients can be computed efficiently by using the FFT, and for smooth functions the resulting approximation is close to optimal, with computations that are numerically stable. Given sufficiently accurate function samples, the Chebyshev expansion coefficients can be computed to machine precision accuracy. However, the accuracy is only with respect to absolute error, and this implies that very small expansion coefficients typically have very large relative error. Upon differentiating a Chebyshev expansion, this relative error in the small coefficients is magnified and accuracy may be lost, especially after repeated differentiation. At first sight, this seems unavoidable. Yet, in this paper, we focus on an alternative computation of Chebyshev expansion coefficients using contour integrals in the complex plane. The main result is that the coefficients can be computed with machine precision relative error, rather than absolute error. This implies that even very small coefficients can be computed with full floating point accuracy,
even when they are themselves much smaller than machine precision. As a result, no accuracy is lost after differentiating the expansion, and even the $100$th derivative of an analytic function can be computed with near machine precision accuracy using standard floating point arithmetic. In some cases, the contour integrals can be evaluated using the FFT, making the approach both highly accurate and fast.
\end{abstract}

{\bf Keywords:} spectral expansion, analytic functions, FFT, spectral differentiation.

\vspace{0.05in}

{\bf AMS classifications:} 42C10, 65N35.

\section{Introduction}\label{sec:introduction}

Among all classical orthogonal polynomials, Chebyshev polynomials play a special rule in numerical analysis due to their connection with FFT algorithms and their numerical stability \cite{trefethen2012atap}. They allow the accurate manipulation of continuous functions using discrete function evaluations \cite{trefethen2011chebfun}. Yet, in spite of the useful connection to the FFT on the real line, the goal of this paper is to show that an alternative computation can make Chebyshev expansions even more accurate than they already are, at least for analytic functions. In many cases, the efficiency of the FFT can be maintained. In a companion paper, we set out to illustrate that non-trivial manipulations of the Chebyshev expansion coefficients, namely their conversion to expansions in more general Jacobi polynomials, maintain this high accuracy beyond what one may have expected.

The contents of this paper have been inspired mainly by a fast method for the computation of Legendre coefficients due to Iserles \cite{iserles2011fastlegendre} and an accurate method for the computation of high-order derivatives in the complex plane due to Bornemann \cite{bornemann2010highorderderivatives}. Before presenting our results, we elaborate briefly on the above references and on other existing research in this area.

\subsection{Fast methods for computing polynomial expansions}

Several methods have been described for the fast computation of
polynomial expansion coefficients
\cite{alpert1991fastlegendre,driscoll1994fouriertransforms,don1994chebyshevlegendre,dutt1996fastpolynomial,driscoll1997fast,potts1998fastpolynomialtransforms,inda2001legendreparallel,keiner2009fastgegenbauer,iserles2011fastlegendre,demicheli2011legendre,cantero2012ultraspherical,xiang2013fastlegendre}.
In particular, the special case of Legendre
polynomials has received the most study
\cite{alpert1991fastlegendre,don1994chebyshevlegendre,inda2001legendreparallel,iserles2011fastlegendre,demicheli2011legendre,xiang2013fastlegendre}.
A popular strategy is to use the FFT, with computations based on
function evaluations at the Chebyshev points
\cite{alpert1991fastlegendre,driscoll1994fouriertransforms,don1994chebyshevlegendre,potts1998fastpolynomialtransforms}.
More general sets of evaluation points have been treated using Fast
Multipole Methods \cite{dutt1996fastpolynomial}, using a particular
matrix-factorization of the problem stated as a matrix-vector
product \cite{driscoll1997fast}, using non-uniform FFT's
\cite{keiner2009fastgegenbauer}
 and using numerical computation of Abel transforms \cite{demicheli2011legendre}. All these methods exhibit $\mathcal{O}(N \log N)$ computational complexity, possibly with additional logarithmic factors, for the computation of the first $N$ coefficients. The accuracy is sometimes restricted to a chosen small value $\epsilon$.

Methods based on function evaluations in the Chebyshev points are, at least mathematically, equivalent to an expansion in Chebyshev
polynomials of the first kind. The coefficients of this expansion can then be rearranged in varying ways in order to form the Legendre
expansion or other expansions. A unique feature of the fast methods for Legendre polynomials in \cite{iserles2011fastlegendre} and more
general ultraspherical polynomials in \cite{cantero2012ultraspherical} is that the evaluation points may be in the complex plane, if the function to be approximated is analytic. We will show further on that this approach is also implicitly equivalent to expanding in a set of Chebyshev polynomials
(of the second kind, in this case), and then rearranging the coefficients. We will be using the same contour integrals as in these references, and a variant which leads to the computation of expansions in Chebyshev polynomials of the first kind.

\subsection{Accurate computation of high-order derivatives}

Computing derivatives of a function numerically is a notoriously ill-conditioned problem, especially for high-order derivatives \cite{miel1985differentiation}. It was shown by Bornemann in \cite{bornemann2010highorderderivatives} that computation of high-order derivatives through Cauchy integrals in the complex plane is, in fact, stable. To be precise, consider the power series of a function analytic at the origin, with radius of convergence $R$,
\[
f(z) = \sum_{k=0}^\infty t_k z^k, \quad |z| < R.
\]
The coefficients $t_n$ can be written as a contour integral along a disc with radius $r < R$,
\begin{equation}\label{eq:taylorcoefficient}
 t_n = \frac{f^{(n)}(0)}{n!} = \frac{1}{2\pi i} \int_{|z|=r} \frac{f(z)}{z^{n+1}} {\rm d}z.
\end{equation}
Such integrals can be evaluated quickly with the FFT for a range of
$n$ with the parametrization $z = r e^{i \theta}$. However,
Bornemann showed that for each $n$ an optimal value of $r$ exists,
such that evaluating the contour integral is, for most analytic
functions, perfectly stable. A detailed analysis is given in
\cite{bornemann2010highorderderivatives} to characterize the optimal
radius, exactly or approximately, for several classes of analytic
functions. Using the optimal radius $r$ for each value of $n$
precludes the use of the FFT. However, small \emph{relative error}
of the coefficient $t_n$ is guaranteed. As a result, for most
analytic functions $t_n$ can be computed with a number of
digits close to the maximal accuracy allowed by the machine
precision and with values of $n$ ranging up to millions.

\subsection{Main results and outline of this paper}

We describe and analyze an efficient way to compute expansions in Chebyshev polynomials of the first or of the second kind. The computations are performed in the complex plane. We use the trapezoidal rule for the integrals
\begin{equation}\label{eq:an}
 a_n = \frac{1}{\pi \rho^n} \int_{0}^{2\pi}  f \left(  \tfrac{1}{2} (\rho e^{i\theta} + ( \rho e^{i\theta})^{-1} )  \right) e^{- i n
\theta} d\theta
\end{equation}
and
\begin{equation}\label{eq:bn}
 b_n = \frac{1}{2\pi \rho^n } \int_{0}^{2 \pi} f \left( \tfrac{1}{2} (\rho e^{i\theta} + ( \rho e^{i\theta} )^{-1} ) \right) (1 - ( \rho e^{i \theta} )^{-2} )  e^{- i n \theta}  d\theta.
\end{equation}
The latter integrals (for $b_n$) are those appearing in
\cite[(3.5)]{iserles2011fastlegendre} and in
\cite[(3.2)]{cantero2012ultraspherical}. We show in
\S\ref{sec:chebyshev expansion} that the values $a_n$ and $b_n$
correspond to coefficients of polynomial expansions using Chebyshev
polynomials of the first and second kind respectively. The value of
$\rho \geq 1$ is arbitrary and limited by the analyticity of $f$.

We show in \S\ref{sec:stability} that the trapezoidal rule for these
integrals is always stable with respect to absolute errors of the
normalized values $\rho^n a_n$ and $\rho^n b_n$. Furthermore, we
show that in many cases for each $n$ an optimal value $\rho^*(n)$
exists, such that the computation is stable with respect to relative
errors. This implies that also very small coefficients can be
computed to high accuracy. The cases depend on the properties of $f$
in the complex plane and they correspond to the cases described by
Bornemann in \cite{bornemann2010highorderderivatives} in the context
of computing high-order derivatives. The integrals \eqref{eq:an} and
\eqref{eq:bn} play the role of the Cauchy integral
\eqref{eq:taylorcoefficient} in
\cite{bornemann2010highorderderivatives}.

We explore two strategies for the computation of Chebyshev coefficients in the complex plane in \S\ref{sec:two strategies}. Efficiency is maximized by using the FFT along a fixed contour in \S\ref{sec:maximize the efficiency}, while accuracy is maximized by optimizing the contour for each coefficient in \S\ref{sec:accuracy}. The theory is illustrated with several numerical examples.

Next, we show that repeated differentiation of the polynomial expansions can be
performed without loss of precision in \S\ref{sec:differentiation}. Finally, we illustrate that this has a beneficial effect on the accuracy of rootfinding, in particular when applying rootfinding on the derivative of a function in order to find its maxima or inflexion points, in \S\ref{sec:roots}. We end the paper with some concluding remarks and questions for further research in \S\ref{sec:conclusion}.

\section{Chebyshev expansion coefficients}\label{sec:chebyshev expansion}

It is well known that the Chebyshev coefficients can be computed efficiently by the FFT and that this computation is numerically stable with respect to absolute errors. In the following, we will show that this strategy remains stable when performing computations along certain contours in the complex plane. For the stability with respect to relative errors, a different theory should be considered. We begin our analysis with an alternative integral expression of Chebyshev coefficients.

\subsection{Chebyshev expansion of the first kind}

Let $T_n(x)$ denote the Chebyshev polynomial of the first kind of
degree $n$, as defined by
\begin{align*}
T_n(\cos\theta) = \cos ( n \theta ), \quad  n \geq 0.
\end{align*}
If a function $f(x)$ satisfies a Dini-Lipschitz condition on
the interval $[-1,1]$ then it can be expanded uniformly in terms of
$T_n(x)$ as \cite[Thm.~5.7]{mason2003chebyshev}
\begin{align}\label{eq:first cheb expansion}
f(x) = \sum_{n = 0}^{\infty}{'}  a_n  T_n(x),
\end{align}
where the prime indicates that the first term of the sum should be
halved and the coefficients are given by the integrals
\begin{align}\label{eq:ChebTcoeff}
a_n = \frac{2}{\pi} \int_{-1}^{1} \frac{ f(x) T_n(x) }{ \sqrt{ 1 -
x^2 } } dx, \quad n\geq 0.
\end{align}
We are interested in integral expressions for $a_n$ in the complex
plane. Let $\mathcal{E}_{\rho}$ denote the \emph{Bernstein ellipse}
\[
 \mathcal{E}_{\rho} =  \left\{ z \in \mathbb{C} ~\bigg| ~z =
 \frac{1}{2} \big( \rho e^{i \theta} + \rho^{-1} e^{-i \theta} \big),~ ~0 \leq
\theta  \leq   2\pi  \right\}.
\]
We will always assume $\rho \geq 1$. 
We denote the interior of this ellipse by
\[
 \mathcal{D}_\rho =  \left\{ z \in \mathbb{C} ~\bigg| ~z =
 \frac{1}{2} \big( r e^{i \theta} + r^{-1} e^{-i \theta} \big),~ ~1 \leq r < \rho, ~0 \leq
\theta  \leq   2\pi  \right\}.
\]
It is well known that the Bernstein ellipses have foci $\pm
1$ and their major and minor semiaxis lengths summing to
$\rho$. In the following, we will often use the notation
\begin{equation}\label{eq:z_u}
 z(u) = \frac{1}{2}(u + u^{-1}),
\end{equation}
where typically $u=\rho e^{i \theta}$ is a point on the circle with radius $\rho$ and $z(u)$ lies on the Bernstein ellipse $\mathcal{E}_{\rho}$. The inverse expression (the one that satisfies $|u|>1$) is
\begin{equation}\label{eq:u_z}
 u(z) = z + \sqrt{z^2-1}.
\end{equation}

The following integral expression for $a_n$ was derived by Elliott in \cite[Eqn.~(28)]{elliott1964chebyshev} for entire functions $f(z)$ by using Cauchy's integral formula. Here, we shall give a simpler proof based on Laurent series expansions. We further show that the expression remains valid for functions analytic only in a neighborhood of the interval $[-1,1]$.

\begin{lemma}
If $f$ is analytic inside and on the Bernstein ellipse
$\mathcal{E}_{\rho}$ with $\rho > 1$, then for each $n \geq 0$ we
have
\begin{align}\label{eq:ChebyT_contour_integral}
a_n & = \frac{1}{\pi \rho^n} \int_{0}^{2\pi}  f \left(  \tfrac{1}{2}
(\rho e^{i\theta} + ( \rho e^{i\theta})^{-1} )  \right) e^{- i n
\theta} d\theta.
\end{align}
\end{lemma}
\begin{proof}
First we recall that the Chebyshev expansion is convergent in
the interior of the greatest ellipse in which $f(x)$ is analytic
\cite[Thm.~9.1.1]{szego1939polynomials}. Moreover, recall the
definition of the Chebyshev polynomials of the first kind in the
complex plane \cite[Eqn.~(1.47)]{mason2003chebyshev}
\begin{align}\label{def:ChebT in complex plane}
T_k( z(u) ) = \frac{1}{2} ( u^k + u^{-k} ),
\end{align}
which implies
\begin{align*}
f ( z(u) ) & = \sum_{n =
0}^{\infty}{'}  a_n  T_n( z(u) ) \\
& = \frac{1}{2} \sum_{n = -\infty}^{\infty} a_{|n|} u^n,
\end{align*}
where $z(u)$ is inside or on the boundary of $\mathcal{E}_{\rho}$.
For each $n\geq 0$, the last equality shows that the $n$-th
Chebyshev coefficient of $f(x)$ corresponds exactly the $n$-th
coefficient of the Laurent series expansion of $2 f(z(u))$ at the
origin. Therefore, we can deduce immediately that for each $n\geq
0$,
\begin{align*}
a_n & = \frac{1}{2 \pi i} \oint_{ \mathcal{C}_{\rho} } 2 f(z(u)) u^{-n-1} d u \\
& = \frac{1}{\pi i} \oint_{ \mathcal{C}_{\rho} } f( z(u) ) u^{-n-1}
d u.
\end{align*}
where $\mathcal{C}_{\rho}$ denotes the circle $ |u| = \rho $.
Substituting $u = \rho e^{i \theta} $ into the last equality yields
the desired result.
\end{proof}

We make some further comments regarding \eqref{eq:ChebTcoeff} and
\eqref{eq:ChebyT_contour_integral}:
\begin{itemize}

\item We define the \emph{normalized} Chebyshev coefficient to be $\rho^n a_n$. In spite of its dependence on the parameter $\rho$, this definition is a natural one because the FFT-based algorithms presented further on yield a small absolute error of the normalized coefficients for a given value of $\rho$.

\item Letting $\rho \rightarrow 1$ and using the change of variable $x =
\cos \theta$, \eqref{eq:ChebyT_contour_integral} reduces to \eqref{eq:ChebTcoeff}.

\item In the same limit $\rho \to 1$, we also obtain the well known expression
\begin{equation}\label{eq:dctformula}
a_n = \frac{1}{\pi} \int_{0}^{2\pi} f(\cos \theta)
e^{-i n \theta} d\theta = \frac{1}{\pi}  \int_{0}^{2 \pi}
f(\cos \theta) \cos(n \theta) d\theta.
\end{equation}
The last expression is often the starting point for
introducing fast algorithms based on the discrete cosine or
Fourier transform to evaluate the Chebyshev coefficients (see, for
example, \cite{trefethen2012atap,gil2007specialfunctions,iserles2009firstcourse,mason2003chebyshev}).

\item Integral expressions for $a_n$ in the complex plane date back at
least to Bernstein \cite{bernstein1912}. They have been used, among other purposes, to estimate the decay
rates of Chebyshev coefficients (see, for example,
\cite{elliott1964chebyshev,rivlin1974chebyshev}). To the best of our knowledge, they have not
been used for computational purposes. One obvious reason is that it is not clear whether there is any advantage in evaluating \eqref{eq:ChebyT_contour_integral} compared to evaluating \eqref{eq:dctformula}, especially in view of the existence of simple, fast and stable algorithms for the latter. Furthermore, expression \eqref{eq:ChebyT_contour_integral} requires analyticity of $f$. We will show later on that expression \eqref{eq:ChebyT_contour_integral} can be used to give better approximations in the sense that the relative error of each Chebyshev coefficient can be minimized by choosing an optimal value of $\rho$.

\item For example, \eqref{eq:ChebyT_contour_integral} leads to the well-known bound \cite[Thm.~3.8]{rivlin1974chebyshev}
\[
 |a_n| \leq \frac{2 {\mathcal{M}}}{\rho^n},
\]
where $\mathcal{M}$ is the maximum absolute value of $f$
along the Berstein ellipse $\mathcal{E}_\rho$.
\end{itemize}

\subsection{Chebyshev expansion of the second kind}\label{ss:secondkind}

Let $U_n(x)$ denote the Chebyshev polynomial of the second kind of
degree $n$, defined by
\begin{align*}
U_n(\cos\theta) = \frac{\sin(n + 1) \theta}{\sin \theta}, \quad n
\geq  0.
\end{align*}
The Chebyshev expansion of the second kind is given by
\begin{align}\label{eq:ChebyU expansion}
f(x) = \sum_{n = 0}^{\infty}  b_n  U_n(x),
\end{align}
where
\begin{align}\label{eq:ChebyU coeff}
b_n = \frac{2}{\pi} \int_{-1}^{1}  \sqrt{ 1 - x^2 }  f(x)  U_n(x)dx.
\end{align}

\begin{lemma}\label{chebyU coefficient}
If $f$ is analytic inside and on the Bernstein ellipse
$\mathcal{E}_{\rho}$ with $\rho>1$, then for each $n \geq 0$ we have
\begin{align}\label{eq:ChebyU coeff contour integral}
b_n & = \frac{1}{2\pi \rho^n } \int_{0}^{2 \pi} f \left(
\tfrac{1}{2} (\rho e^{i\theta} + ( \rho e^{i\theta} )^{-1} ) \right)
(1 - ( \rho e^{i \theta} )^{-2} )  e^{- i n \theta}  d\theta.
\end{align}
\end{lemma}
\begin{proof}
Using the definition of the Chebyshev polynomials of the
second kind in the complex plane \cite[Eqn.~(1.51)]{mason2003chebyshev}
\[
U_k( z(u) ) = \frac{u^{k+1} - u^{-k-1}}{u - u^{-1}},
\]
we have that
\begin{align*}
f ( z(u) ) & = \sum_{n =
0}^{\infty}  b_n  U_n( z(u) ) \\
& =  \sum_{n = 0}^{\infty} b_{n} \frac{u^{n+1} - u^{-n-1}}{u -
u^{-1}}.
\end{align*}
Multiplying both sides of the last equality by $1 - u^{-2}$ gives
\begin{align*}
f ( z(u) ) ( 1- u^{-2} ) & =  \sum_{n = 0}^{\infty} b_{n} ( u^{n} -
u^{-n-2} ).
\end{align*}
For each $n\geq 0$, the above equality shows that $b_n$ corresponds
exactly to the $n$-th coefficient of the Laurent series expansion of $
f(z(u)) (1-u^{-2})$ at the origin. Therefore, we can deduce
immediately that for each $n\geq 0$,
\begin{align}\label{eq:ChebyU coeff contour}
b_n & = \frac{1}{2 \pi i} \oint_{ \mathcal{C}_{\rho} } f ( z(u) ) (
1- u^{-2} ) u^{-n-1} d u .
\end{align}
Substituting $u = \rho e^{i \theta} $ into the last equality yields
the desired result.
\end{proof}

Here, too, we will make some further comments regarding \eqref{eq:ChebyU coeff contour integral}:

\begin{itemize}
\item Similarly as before, we define the \emph{normalized} Chebyshev coefficient to be $\rho^n b_n$.

\item Letting $\rho \rightarrow 1$ in \eqref{eq:ChebyU coeff contour
integral} yields,
\begin{align}
b_n & = \frac{1}{2\pi } \int_{0}^{2\pi} f(\cos\theta)(1 -
e^{-2i\theta} ) e^{-i n \theta}  d\theta  \nonumber \\
& = \frac{1}{\pi } \int_{0}^{2\pi} f(\cos\theta)  \sin \theta
\sin(n + 1)\theta  d\theta  \nonumber \\
& = \frac{2}{\pi} \int_{-1}^{1} \sqrt{1 - x^2} f(x) U_n(x) dx,
\end{align}
which corresponds to \eqref{eq:ChebyU coeff}.

\item Expression \eqref{eq:ChebyU coeff contour} can be further written as
\begin{align}
b_n & = \frac{1}{2 \pi i} \oint_{ \mathcal{C}_{\rho} } f ( z(u) ) (
1- u^{-2} ) u^{-n-1} d u \nonumber \\
& =  \frac{1}{\pi i}  \oint_{ \mathcal{E}_{\rho} }f(z(u)) u(z)^{- n
- 1 } d z(u),
\end{align}
which can be used to established the rate of decay of the
coefficients $b_n$.

\item From the inspection of the formulas \eqref{eq:ChebyT_contour_integral} and \eqref{eq:ChebyU coeff contour integral}, it is clear that $a_n$ and $b_n$ are related for all $\rho$ by
\begin{equation}\label{eq:ab_relationship}
 b_n = \frac{a_n - a_{n+2}}{2}.
\end{equation}

\end{itemize}


\section{Absolute and relative stability}\label{sec:stability}

From \eqref{eq:ChebyT_contour_integral} and \eqref{eq:ChebyU coeff
contour integral} we see that both kinds of Chebyshev coefficients
can be expressed in terms of contour integrals with integrands that
are periodic functions of $\theta$. Thus, these coefficients can be
approximated efficiently by applying the trapezoidal rule. For
a more detailed and theoretical analysis of the trapezoidal rule for
periodic and analytic functions, we refer the reader to
\cite{trefethen2014trapezoidal}. In the following we shall consider
stability of the computation of the Chebyshev coefficients with
respect to absolute and relative errors of the normalized
coefficients, respectively.

\subsection{Absolute stability}

For the Chebyshev coefficients of the first kind, using an $m$-point
trapezoidal rule yields
\begin{align}\label{eq:trap for cheb1}
a_{n}(m,\rho) = \frac{2}{m \rho^n} \sum_{j = 0}^{m - 1}  f \left(
\tfrac{1}{2} (\rho e^{2\pi ij/m} + \rho^{-1}e^{-2\pi ij/m} ) \right)
e^{-2\pi i j n/m}.
\end{align}
Let $\mathcal{P}_m$ be the set of all polynomials of degree $\leq m$
and let
\begin{align*}
\| f - s \|_{ \mathcal{D}_{\rho} } := \max_{z \in
\overline{\mathcal{D}}_{\rho} }  | f(z)  -  s(z) |.
\end{align*}
Note that by the maximum modulus principle we have the equality of norms
\begin{equation*}
 \| f - s \|_{ \mathcal{D}_{\rho} } = \| f - s \|_{ \mathcal{E}_{\rho} } := \max_{z \in {\mathcal E}_\rho} | f(z) - s(z)|,
\end{equation*}
so that from now on we simply use $\Vert \cdot \Vert_{\mathcal{E}_{\rho}}$.


Furthermore, let
\[
s_{m}(z) = \sum_{ k = 0}^{m - 1}{'} \eta_k T_k(z)
\]
denote the best $(m-1)$-th degree polynomial approximation to $f(z)$
on and inside the ellipse $\mathcal{E}_{\rho}$, i.e.,
\[
\| f(z) - s_m(z) \|_{ \mathcal{E}_{\rho}  } := \inf_{s \in
\mathcal{P}_{m-1} } \| f(z) - s(z) \|_{  \mathcal{E}_{\rho}  }.
\]

In the following, we will always assume that the \emph{sampling condition} $m > n$ holds, in order to avoid aliasing of the complex exponentials in \eqref{eq:trap for cheb1}. We refer the reader to \cite[\S2.1]{bornemann2010highorderderivatives} for a discussion and justification of this condition.

\begin{theorem}\label{thm:absolute stability}
For $1 \leq  n < m$, we have the following error estimate
\begin{align}
| a_n  -  a_{n}(m,\rho) | \leq  \frac{4}{\rho^n} \| f - s_{m} \|_{
\mathcal{E}_{\rho} } + \frac{ |\eta_{m-n} | }{\rho^m},
\end{align}
and for $n = 0$,
\begin{align}
| a_0  -  a_{0}(m,\rho) | \leq   4 \| f - s_{m} \|_{
\mathcal{E}_{\rho} }.
\end{align}
\end{theorem}
\begin{proof} Let $z(\rho, \theta) = \frac{1}{2} ( \rho e^{i\theta} + (\rho e^{i \theta})^{-1} )$.
From \eqref{eq:ChebyT_contour_integral} and \eqref{eq:trap for cheb1}, we have
\begin{align*}
a_n - a_n(m,\rho) & =  \frac{1}{\pi \rho^n} \int_{0}^{2\pi}  f
\left( z(\rho, \theta) \right) e^{- i n \theta} d\theta  - \frac{2}{m \rho^n} \sum_{j = 0}^{m - 1} f ( z(\rho, 2 \pi j /m ) ) e^{-2\pi i j n/m} \\
& = \frac{1}{\pi \rho^n} \int_{0}^{2\pi} [ f \left( z(\rho, \theta)
\right) - s_m \left( z(\rho, \theta) \right) ] e^{- i n \theta} d\theta \\
&~~~~~ + \left(  \frac{1}{\pi \rho^n} \int_{0}^{2\pi}  s_m \left(
z(\rho, \theta) \right) e^{- i n \theta} d\theta  - \frac{2}{m
\rho^n} \sum_{j = 0}^{m - 1} s_m ( z(\rho, 2 \pi j /m ) ) e^{-2\pi i
j n/m}
\right) \\
&~~~~~ + \frac{2}{m\rho^n} \sum_{j =0}^{m-1} [s_m ( z(\rho, 2
\pi j /m ) ) - f ( z(\rho, 2 \pi j /m ) )] e^{-2\pi i j n/m}.
\end{align*}
We use $E_1$ to denote the first integral of the last equality,
$E_2$ denotes the difference contained in the brackets and $E_3$
denotes the remaining part. Explicit estimates can be established
for $E_1$ and $E_3$,
\begin{align*}
| E_1 | \leq \frac{2}{\rho^n} \| f - s_{m} \|_{ \mathcal{E}_{\rho}
}, \quad  | E_3 |  \leq  \frac{2}{\rho^n} \| f - s_{m} \|_{
\mathcal{E}_{\rho} }.
\end{align*}
For $E_2$, using \eqref{def:ChebT in complex plane}
we have
\begin{align*}
E_2 & =  \frac{1}{\pi \rho^n} \int_{0}^{2\pi}  s_m \left( z(\rho,
\theta) \right) e^{- i n \theta} d\theta  - \frac{2}{m \rho^n}
\sum_{j =
0}^{m - 1} s_m ( z(\rho, 2 \pi j /m ) ) e^{-2\pi i j n/m} \\
& = \eta_n - \frac{1}{m \rho^n} \sum_{k = -(m-1)}^{m-1} \eta_{|k|}
\rho^k \left(  \sum_{j=0}^{m-1} e^{2 \pi i j (k-n)/m }  \right) \\
& =  \left\{\begin{array}{cc}
                                           0 ,                     & \hbox{$\textstyle  n =
                                           0$}, \\ [6pt]
                                           - \frac{\eta_{m-n}}{ \rho^m },    & \hbox{$1 \leq n < m$}.
                                        \end{array}\right.
\end{align*}
Combining this with estimates of $E_1$ and $E_3$ gives the desired
results.
\end{proof}

From Theorem \ref{thm:absolute stability} we can see that if $f$ is
a polynomial of degree $n$, then we have $s_m = f$ if $m \geq n +
1$. This implies that the trapezoidal rule \eqref{eq:trap for cheb1}
computes the $k$-th Chebyshev coefficient of $f$ exactly if $m \geq
k + n + 1$ since $\eta_{m-n} = 0$. Thus, if we choose $m \geq 2n+1$,
then all Chebyshev coefficients of the polynomial function $f$ can
be computed exactly by the trapezoidal rule \eqref{eq:trap for
cheb1}.


Theorem \ref{thm:absolute stability} implies for any function $f$ that the difference in the normalized coefficients $\rho^n a_n -\rho^n a_n(m,\rho)$ is on the order of $\epsilon$, if $m$ is sufficiently large so that $\eta_{m-n}$ is small.
This assertion is true, since from
\begin{align*}
|\eta_k - a_k| & = \left| \frac{1}{\pi i} \oint_{ \mathcal{C}_{\rho}
} ( s_m( z(u) ) - f( z(u) ) ) u^{-n-1} d u \right| \\
& \leq \frac{2}{\rho^k} \| f(z) - s_m(z) \|_{ \mathcal{E}_{\rho} },
\end{align*}
it follows that
\begin{align*}
|\eta_k | & \leq  | a_k | + \frac{2}{\rho^k} \| f(z) - s_m(z) \|_{
\mathcal{E}_{\rho} } \\
& \leq  \frac{2}{\rho^k} (  {\mathcal{M}} + \| f(z) - s_m(z) \|_{
\mathcal{E}_{\rho} } ).
\end{align*}
This estimate implies that the coefficients $\eta_k$ decay
exponentially fast.

Similarly, for the Chebyshev coefficients $b_n$, the $m$-point
trapezoidal rule gives
\begin{align}\label{eq:trap for cheb2}
b_{n}(m,\rho) = \frac{1}{m \rho^n} \sum_{j = 0}^{m - 1}  f \left(
z(\rho, 2 \pi j /m ) \right) (1 - \rho^{-2} e^{-4 \pi ij/m} )
e^{-2\pi i j n/m}.
\end{align}

\begin{theorem}
We have the following error estimate
\begin{align}
| b_n  -  b_{n}(m,\rho) | \leq  \frac{2 (1 - \rho^{-2}) }{\rho^n} \|
f - s_{m} \|_{ \mathcal{E}_{\rho} } + \left\{\begin{array}{cc}
                                          \frac{| \eta_{m - 2} {} |}{2 \rho^m },                     & \hbox{$\textstyle  n=0$},\\
                                           \frac{| \eta_{m - n} - \eta_{m - n - 2} |}{ 2\rho^m },    & \hbox{$\textstyle  n = 1, \ldots, m-3$},\\
                                          \frac{| \eta_{m - n} - \eta_{ n + 2 - m} |}{ 2\rho^m },    & \hbox{$\textstyle  n = m-2, m-1  $ }.
                                        \end{array}\right.
\end{align}

\end{theorem}
\begin{proof}
The proof is essentially the same as that of Theorem
\ref{thm:absolute stability}. We omit the details.
\end{proof}
Similarly to \eqref{eq:trap for cheb1}, if $f$ is a polynomial
of degree $n$, then $b_k$ is computed exactly by the trapezoidal
rule \eqref{eq:trap for cheb2} if $m \geq k +n + 3$. This implies
that all $\{b_k\}_{k=0}^{n}$ are computed exactly by the
trapezoidal rule \eqref{eq:trap for cheb2} if we choose $m \geq
2n+3$.

Suppose now that $\hat{f}$ is a perturbation of $f$ and
\[
\| \hat{f}(z) - f(z)  \|_{ \mathcal{E}_{\rho} } \leq \epsilon.
\]
The perturbed Chebyshev coefficients are given by
\begin{align}
\hat{a}_{n} =  \frac{1}{\pi \rho^n} \int_{0}^{2\pi}  \hat{f} \left(
\tfrac{1}{2} (\rho e^{i\theta} +  ( \rho e^{i\theta})^{-1} )
\right) e^{- i n \theta} d\theta.
\end{align}
Meanwhile, the computed Chebyshev coefficients are given by
\begin{align}
\hat{a}_{n}(m,\rho) =  \frac{2}{m \rho^n} \sum_{j = 0}^{m - 1}
\hat{f} \left( \tfrac{1}{2} (\rho e^{2\pi ij/m} + \rho^{-1}e^{-2\pi
ij/m} ) \right) e^{-2\pi i j n/m}.
\end{align}
A simple bound can be derived for the Chebyshev coefficients of the first kind
\begin{align}
| a_n - \hat{a}_{n} | \leq \frac{ 2 \epsilon }{ \rho^n }, \quad   |
\hat{a}_{n}(m,\rho) - a_{n}(m,\rho) | \leq \frac{ 2 \epsilon
}{\rho^n} .
\end{align}
Then the following estimate also holds
\begin{align*}
\rho^n | \hat{a}_n(m,\rho) -  a_n  | & \leq \rho^n |a_n(m,\rho) - \hat{a}_n(m,\rho) | + \rho^n | a_n - a_n(m,\rho) |  \\
& \leq 2 \epsilon  +  \left\{\begin{array}{cc}
                                           4 \| f - s_{m} \|_{
\mathcal{E}_{\rho} } ,                     & n = 0,\\
                                           4 \| f - s_{m} \|_{
\mathcal{E}_{\rho} } + \frac{ |\eta_{m-n} | }{\rho^{m-n}},    & 1\leq n < m.
                                        \end{array}\right.
\end{align*}
A similar estimate can be established for the coefficients of the second kind $b_n$.

We conclude that the trapezoidal rule for the Chebyshev coefficients
is numerically stable with respect to the absolute error of the
normalized coefficients. If we only consider this absolute
stability, then it is sufficient to choose the same $\rho$
simultaneously for all Chebyshev coefficients and to compute these
coefficients with the same trapezoidal rule. Furthermore, from
\eqref{eq:trap for cheb1} we see that the sum on the right hand side
is perfectly suitable to utilize the FFT. Thus, the first $N$
Chebyshev coefficients can be efficiently evaluated with a single
FFT in $\mathcal{O}( N\log N)$ operations.

\subsection{Relative stability}

If we consider the relative error of the computed coefficients, computing all Chebyshev coefficients with a single $\rho$ is not optimal. A comprehensive analysis of
the relative stability of computing the Taylor expansion
coefficients of analytic functions from contour integrals along
circles in the complex plane has been given by Bornemann in
\cite{bornemann2010highorderderivatives}. Here we extend his analysis to the current
setting of Chebyshev coefficients.

Suppose $\hat{f}$ is a perturbation of $f$ with the form
\[
\hat{f}(z) = f(z) (1 + \epsilon_\rho(z) ), \quad
|\epsilon_\rho(z)| \leq \epsilon.
\]
There is a simple upper bound on the error of the perturbed Chebyshev coefficients,
\begin{align*}
| a_n - \hat{a}_n |  & =  \frac{1}{ \pi \rho^n }  \left|
\int_{0}^{2\pi} f\left( \tfrac{1}{2} (\rho e^{i\theta} +  ( \rho
e^{i\theta})^{-1} ) \right) \epsilon_\rho \left( \tfrac{1}{2} (\rho
e^{i\theta} +  ( \rho e^{i\theta})^{-1} ) \right)
e^{-in\theta} d\theta \right| \\
& \leq \frac{ \epsilon }{ \pi \rho^n } \int_{0}^{2\pi} \left| f
\left( \tfrac{1}{2} (\rho e^{i\theta} +  ( \rho e^{i\theta})^{-1} )
\right) \right|  d\theta,
\end{align*}
which leads to
\begin{align}
\frac{ | a_n - \hat{a}_n | }{ |a_n| } \leq
\kappa^{\mathrm{Ch1}}(n,\rho) \epsilon,
\end{align}
where the quantity
\begin{align}
\kappa^{\mathrm{Ch1}}(n, \rho) = \frac{ \int_{0}^{2\pi} \left| f
\left(\tfrac{1}{2}(\rho e^{i\theta} + (\rho e^{i\theta})^{-1}
)\right)  \right|  d\theta }{ | \int_{0}^{ 2 \pi } f \left(
\frac{1}{2} (\rho e^{i\theta} + ( \rho e^{i \theta } )^{-1} )
\right) e^{-i n \theta} d \theta |} \geq 1,
\end{align}
is called the condition number of the integral. Similarly, for the Chebyshev coefficients of the second kind, we have
\begin{align}
\frac{ | b_n - \hat{b}_n | }{ |b_n| } \leq
\kappa^{\mathrm{Ch2}}(n,\rho) \epsilon,
\end{align}
with the corresponding condition number given by
\begin{align}
\kappa^{\mathrm{Ch2}}(n, \rho) = \frac{ \int_{0}^{2\pi} \left| f
\left(\tfrac{1}{2}(\rho e^{i\theta} + (\rho e^{i\theta})^{-1}
)\right)  (1 - (\rho e^{i\theta} )^{-2} ) \right| d\theta }{ |
\int_{0}^{ 2 \pi } f \left( \frac{1}{2} (\rho e^{i\theta} + (\rho
e^{i \theta })^{-1} ) \right) (1 - (\rho e^{i\theta} )^{-2} ) e^{-i
n \theta} d \theta |}.
\end{align}

\subsection{Condition number of the contour integrals}

We consider the condition number of the integral expressions for the Chebyshev coefficients
of the first kind. The corresponding integrals for the Chebyshev coefficients of the second kind can
be analyzed similarly.

We first rewrite the condition number as
\begin{equation}\label{eq:condition number ch1}
\kappa^{\mathrm{Ch1}}(n, \rho) = \frac{ M(\rho) }{ |a_n| \rho^{n} },
\end{equation}
where
\begin{equation}\label{eq:M}
M(\rho) = \frac{1}{\pi} \int_{0}^{2\pi} \left| f
\left(\tfrac{1}{2}(\rho e^{i\theta} + (\rho e^{i\theta})^{-1}
)\right) \right|  d\theta.
\end{equation}
Note that $M(\rho) = M(\rho^{-1})$.

We proceed by analyzing this function $M(\rho)$. It is the analogue of the function
\begin{equation}\label{eq:M1}
 M_1(r) = \frac{1}{2\pi} \int_0^{2\pi} | f(r e^{i\theta})| d\theta,
\end{equation}
which appears in the condition number for the Cauchy integral \eqref{eq:taylorcoefficient} in the analysis of Bornemann.
He showed that $M_1(r)$ has a unique minimum at a finite value of $r$. The starting point of this analysis is a theorem on the growth of $M_1(r)$ \cite[Thm 4.1]{bornemann2010highorderderivatives} originally due to Hardy in 1915 \cite{hardy1915}. Unfortunately, Hardy's original proof for $M_1(r)$ does not apply for the analysis of the function $M(\rho)$, since the integrand of \eqref{eq:M} is not analytic at the origin. In the following theorem we formulate the corresponding result for $M(\rho)$, with a method of proof that still largely follows that of Hardy.


\begin{theorem}\label{thm:modulus of f}
Let $f$ be analytic in any ellipse $\mathcal{E}_{\rho} $ with $1
\leq \rho < R$. The function $M(\rho)$ satisfies the following
properties:
\begin{enumerate}
\item  $M(\rho)$ is continuously differentiable.
\item  If $f$ is not a constant, $M(\rho)$ is increasing as $\rho$ grows.
\item  If $f \not\equiv 0$, then $\log M(\rho) $ is a convex function of $\log \rho $.
\end{enumerate}
\end{theorem}
\begin{proof}
Let $g(z) = f(\frac{1}{2}(z + z^{-1}))$ and note that $g(z)$ is analytic in the annulus $R^{-1} < |z| < R$. Hence,
$M(\rho)$ can be rewritten as
\[
M(\rho) = \frac{1}{\pi} \int_{0}^{2\pi} |g(\rho e^{i\theta})|
d\theta.
\]
We further define $|g(\rho e^{i\theta})| = g(\rho e^{i\theta})
\varphi(\rho, \theta)$ and
\[
F(z) = \frac{1}{\pi} \int_{0}^{2\pi} g(z e^{i\theta}) \varphi(\rho,
\theta) d\theta.
\]
It is clear to see that $F(z)$ is analytic in the annulus $R^{-1} <
|z| < R$. For $1 \leq \rho \leq r$ and $1 \leq r < R$, we restrict
our attention to the annulus $r^{-1}<|z|<r$. By the maximum modulus
theorem, $F(z)$ achieves its maximum modulus on the boundary $|z| =
r^{-1}$ or $|z| = r$. More specifically, we suppose that $F(z)$
achieves its maximum modulus at $z = r^{-1} e^{i\theta_1}$ or $F(z)
= r e^{i\theta_2}$. Therefore,
\begin{align}
M(\rho) = F(\rho) & \leq \max \left\{ |F(r^{-1} e^{i\theta_1})|,
|F(r e^{i\theta_2})| \right\} \nonumber \\
& \leq \max\left\{ M(r^{-1}), M(r) \right\} \nonumber \\
& = M(r),
\end{align}
where we have used the fact that $|\varphi(\rho, \theta)| = 1$. This
proves the second assertion. For the first and the third assertions,
noting that Hardy's proof given in \cite{hardy1915} is still valid for
functions $g$ defined on an annulus region, these two assertions
follow immediately.
\end{proof}

Since $\log \kappa^{\mathrm{Ch1}}(n,\rho) = \log M(\rho) - \log |a_n| - n \log \rho$, we have the following corollary.

\begin{corollary}\label{cor:modulus of f}
Let $f$ be analytic on and inside an ellipse $ \mathcal{E}_{\rho} $ with $1
\leq \rho < R$. Then for each Chebyshev coefficient $a_n \neq 0$, we have
\begin{enumerate}
\item  $\kappa^{\mathrm{Ch1}}(n, \rho)$ is continuously differentiable with respect to $\rho$.
\item  If $f$ is not a constant, $\log ( \kappa^{\mathrm{Ch1}}(n, \rho) )$ is a convex function of $\log \rho $.
\end{enumerate}
\end{corollary}


In the analysis of Bornemann, \cite[Theorem 4.1]{bornemann2010highorderderivatives} and \cite[Corollary 4.2]{bornemann2010highorderderivatives} are the key steps in proving that an optimal radius exists for Cauchy integrals of the form \eqref{eq:taylorcoefficient}. Afterwards, it remains to analyze the limits $r \to 0$ and $r \to \infty$. The limit $r \to 0$ is always unstable. The limit in the other direction depends on the analyticity properties of $f$ in the complex plane.

%

With our analogous Theorem \ref{thm:modulus of f} and Corollary \ref{cor:modulus of f} at hand, we can reuse Bornemann's results in the context of Chebyshev coefficients with only slight adjustments. One major difference concerns the difference between the limits for small $\rho$ and $r$. Indeed, contrary to the limit $r \to 0$ in the setting of Taylor series coefficients, there is no numerical instability associated with the limit $\rho \to 1$. Recall also that $M(\rho) = M(\rho^{-1})$, so that we don't consider the case $\rho < 1$. It is clear that $M(\rho)$ is bounded as $\rho \to 1$ and we have:

\begin{theorem}\label{thm:limit small rho}
 Assume $f$ is analytic in any ellipse $\mathcal{E}_\rho$ with $1 \leq \rho < R$ and let $a_{n}$ be nonzero. Then
 \[
  \lim_{\rho \to 1} \kappa^{\mathrm{Ch1}}(n,\rho) = \frac{1}{\pi |a_n|} \int_0^{2\pi} |f(\cos \theta)| d\theta.
 \]
\end{theorem}
\begin{proof}
 This follows from the definitions \eqref{eq:condition number ch1} and \eqref{eq:M}.
\end{proof}

Two interesting results to formulate explicitly are as follows.

\begin{theorem}\label{thm:optimal radius for entire functions}
Assume $f$ is an entire transcendental function and
\begin{align}\label{eq:asymptotic of M}
M(\rho) \sim e^{\mu \rho^{\nu} } \rho^{\varsigma}, \qquad \rho \rightarrow\infty,
\end{align}
where $\mu$ is positive and finite and $\nu$ is positive.
Then, the optimal radius satisfies asymptotically
\begin{align}\label{eq:optimal radius}
\rho^{*}(n) \sim \left( \frac{n - \varsigma}{\mu\nu}
\right)^{ \frac{1}{\nu} }.
\end{align}
\end{theorem}
\begin{proof}
For large $\rho$, we have the asymptotic behaviour of the condition number
\begin{align*}
\kappa^{\mathrm{Ch1}}(n, \rho) & = \frac{ M(\rho) }{ |a_n|
\rho^{n} } \sim \frac{e^{\mu \rho^{\nu} } \rho^{\varsigma-n}}{|a_n|}
= \frac{1}{|a_n|} e^{\mu \rho^{\nu} + (\varsigma-n) \log\rho }.
\end{align*}
According to Theorem \ref{thm:modulus of f}, we can simply
differentiate the expression on the right hand side to derive an
optimal value of $\rho$ such that the condition number
$\kappa(n,\rho)$ is asymptotically minimized. This formal
differentiation of an asymptotic formula is guaranteed to be
valid in this case: for a rigorous discussion, we refer the reader to
\cite[Thm.~8.4]{bornemann2010highorderderivatives}. Direct
calculation shows that the above asymptotic expression on the right
hand side takes its minimum value at $\rho = \left( \frac{n -
\varsigma}{\mu\nu} \right)^{ \frac{1}{\nu} }$. This completes the
proof.
\end{proof}

Next, we consider the case where $f$ is only analytic in a bounded region in the complex plane. Define
\[
\vartheta = \sup_{1 < \rho < \rho_{\max}}\frac{ \rho M{'}(\rho)}{
M(\rho) }.
\]
Furthermore, applying the third assertion of Theorem \ref{thm:modulus of f}, we have
\[
\vartheta = \lim_{ \rho \rightarrow \rho_{\max} } \frac{ \rho
M{'}(\rho)}{ M(\rho) }.
\]
The following theorem is analogous to \cite[Thm.~4.5]{bornemann2010highorderderivatives}, which shows the
optimal radius approaches $\rho_{\max}$ for large $n$.
\begin{theorem}\label{thm:optimal for finite radius}
Let $f$ be analytic in any ellipse $\mathcal{E}_{\rho} $ with $1
\leq \rho < R < \infty$. Then,
\begin{enumerate}
\item  If $n > \vartheta$, the condition number $\kappa^{\mathrm{Ch1}}(n,
\rho)$ is strictly decreasing for $1 < \rho < \rho_{\max}$.
\item  If $\vartheta
= \infty$, then $\kappa^{\mathrm{Ch1}}(n, \rho)$ is strictly
increasing in the vicinity of $\rho = \rho_{\max}$.
\item  If $\vartheta < \infty $
and $ \lim_{\rho \rightarrow \rho_{\max} } M(\rho) $ exists
and is finite, then the optimal radius $\rho = \rho_{\max}$ for $n
> \vartheta$.
\end{enumerate}
\end{theorem}
\begin{proof}
In analogy to \cite[Thm.~4.5]{bornemann2010highorderderivatives},
differentiating the condition number with respect to $\rho $ yields
\begin{align*}
\frac{d}{d\rho} \log \kappa^{\mathrm{Ch1}}(n, \rho) =
\frac{M{'}(\rho)}{ M(\rho) }  - \frac{n}{\rho} \leq \frac{\vartheta
- n}{\rho}.
\end{align*}
If $n > \vartheta$, then the condition number
$\kappa^{\mathrm{Ch1}}(n, \rho)$ is a strictly decreasing function
of $\rho$ and the first assertion follows. If $\vartheta = \infty$,
this implies that $\kappa^{\mathrm{Ch1}}(n, \rho)$ is strictly
increasing when $\rho \rightarrow \rho_{\max}$, thus the second
assertion holds. Finally, if $\vartheta < \infty $ and $ \lim_{\rho
\rightarrow \rho_{\max} } M(\rho) $ exists and is finite,
then the third assertion follows from the first assertion.
\end{proof}

\subsection{Examples of optimal contours}


In this section we give some specific examples of optimal radii.
However, first we show that the condition number accurately predicts the relative error of the Chebyshev coefficients. Fig. \ref{fig:condition number for entire} shows the condition number, as well as the ratio of the relative error of the Chebyshev coefficients to the machine precision, for two
entire functions $f(x) = e^x$ and $f(x)=\cos(2x+2)$. There is a clear agreement between both quantities. Fig. \ref{fig:condition number for pole} shows the same experiment for two analytic functions that are not entire, $f(x) = \frac{1}{x-2}$ and $f(x)=\frac{x+1}{x^2+1}$. From this figure we observe that the condition number assumes its minimum value when $\rho$ is close to its maximum value.

\begin{figure}[h]
\centering
\begin{overpic}
[width=6.5cm]{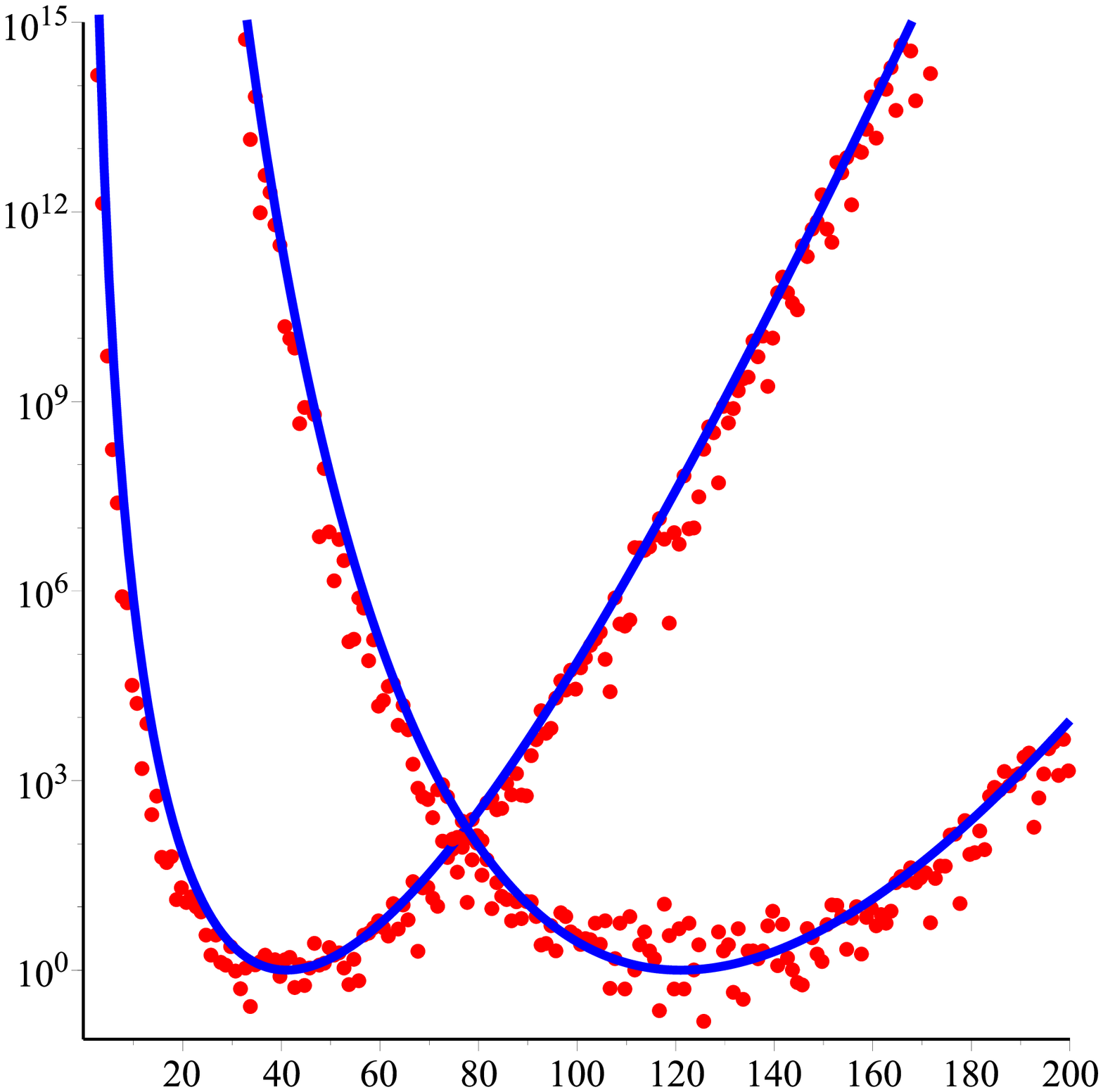}
\end{overpic}
\qquad
\begin{overpic}
[width=6.5cm]{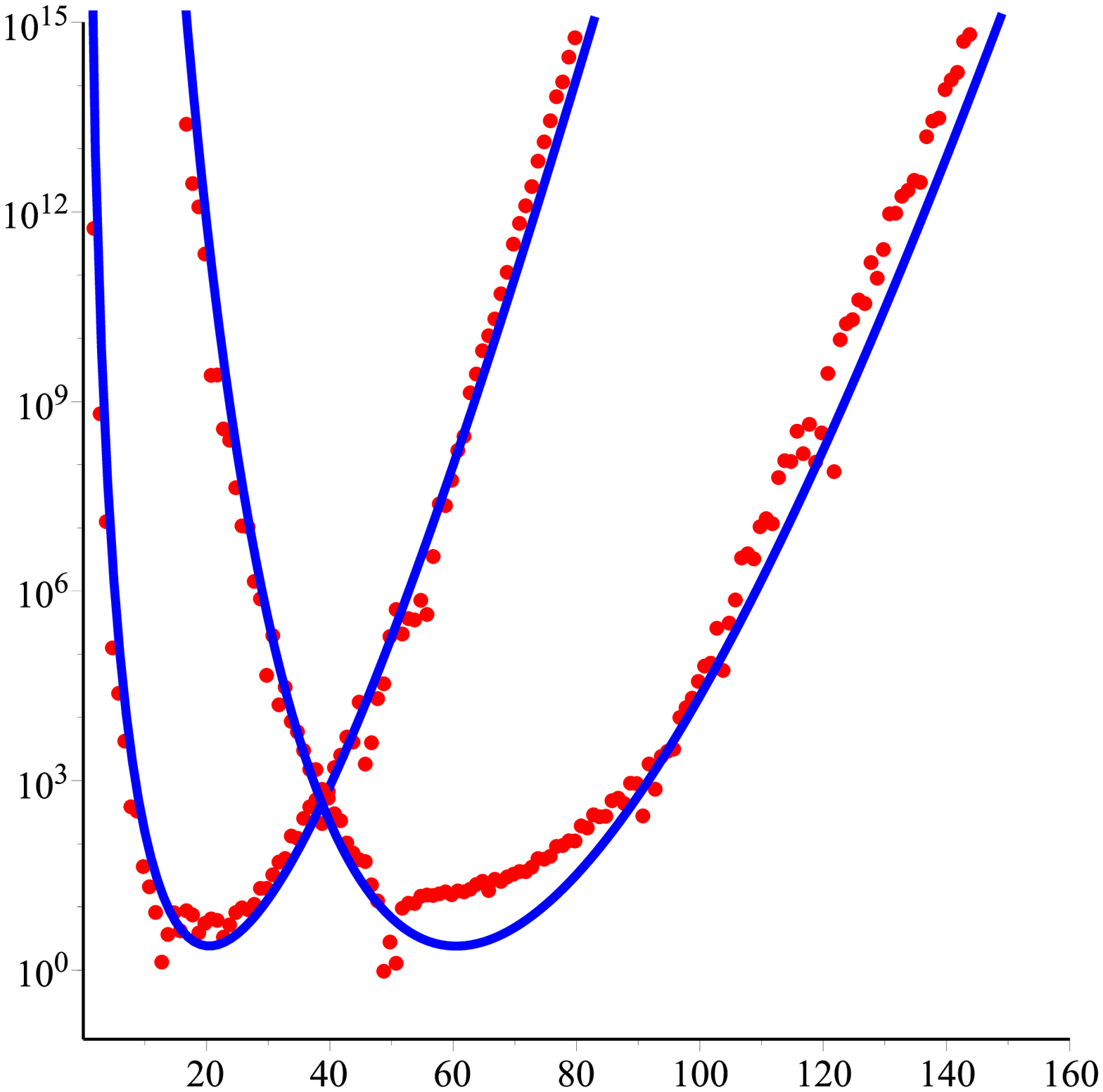}
\end{overpic}
\caption{Ratio of the relative error of the $n$-th Chebyshev
coefficients to the machine precision (dots) and the condition
number $\kappa(n,\rho)$ (line) for $n = 20, 60$, respectively. The
test functions are $f(x) = e^x$ (left) and $f(x) = \cos(2x+2)$
(right).} \label{fig:condition number for entire}
\end{figure}

\begin{figure}[h]
\centering
\begin{overpic}
[width=6.5cm]{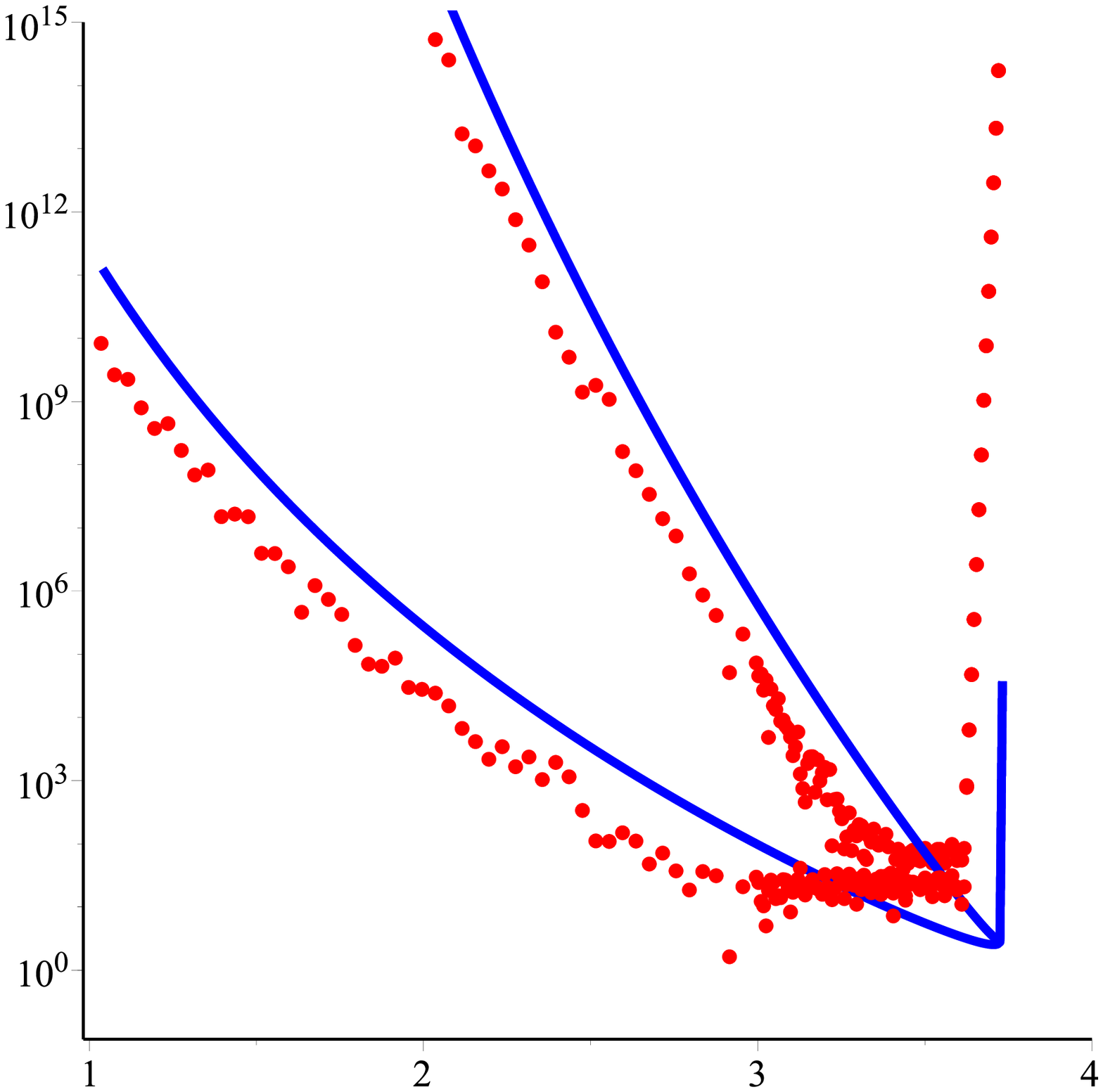}
\end{overpic}
\qquad
\begin{overpic}
[width=6.5cm]{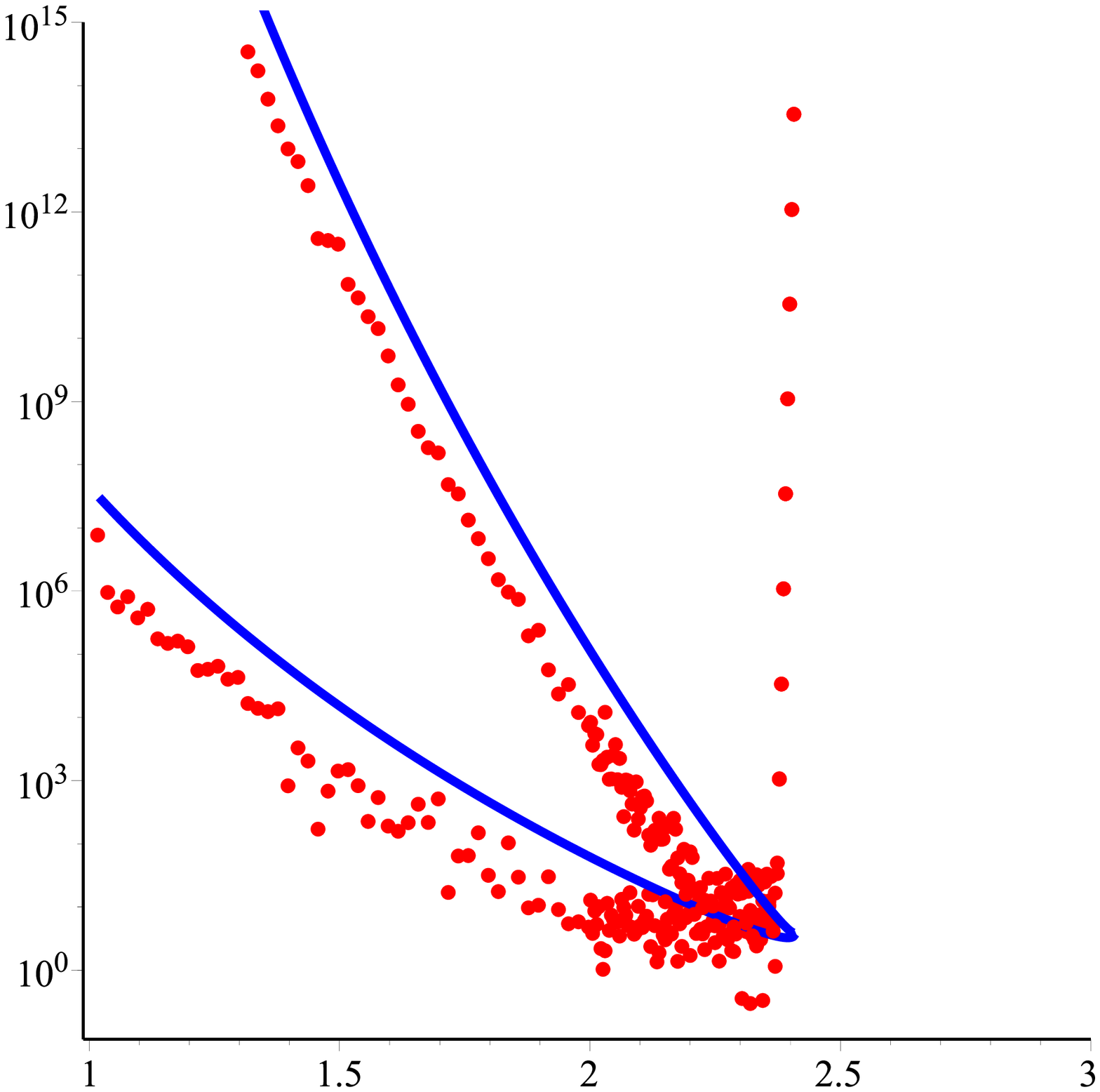}
\end{overpic}
\caption{Ratio of the relative error of the $n$-th Chebyshev
coefficients to the machine precision (dots) and the condition
number $\kappa(n,\rho)$ (line) for $n = 20, 60$, respectively. The
test functions are $f(x) = \frac{1}{x-2}$ (left) and $f(x) = \frac{x
+ 1}{x^2 + 1}$ (right). } \label{fig:condition number for pole}
\end{figure}

\begin{example}
Consider the exponential function $f(x) = e^x$, which is entire and transcendental. Its Chebyshev
coefficients are $a_n = 2 I_n(1)$ and
\[
\int_{0}^{2 \pi} \left| e^{ \frac{1}{2} (\rho e^{i\theta} + (\rho
e^{i\theta})^{-1} ) } \right| d\theta  = \int_{0}^{2\pi} e^{
\frac{1}{2} (\rho + \rho^{-1} ) \cos\theta  } d\theta = 2 \pi
I_0(\tfrac{1}{2} ( \rho + \rho^{-1} ) ),
\]
where $I_n(x)$ is the modified Bessel function of the first kind of
order $n$ \cite[p.~376]{abramowitz1965handbook}. Thus, the condition number is
\begin{align}
\kappa^{\mathrm{Ch1}}(n, \rho)  =  \frac{1}{I_n(1)} I_0(
\tfrac{1}{2} (\rho + \rho^{-1}) ) \rho^{-n}.
\end{align}
Using the first term of the asymptotic expansion of the $I_n(x)$
\cite[p.~377]{abramowitz1965handbook}
\[
I_n(x) = \frac{e^x}{\sqrt{2\pi x}} \left( 1 - \frac{4n^2 - 1}{8x} +
\mathcal{O}(x^{-2}) \right), \quad  x \rightarrow\infty,
\]
we get from Theorem \ref{thm:optimal radius for entire functions} that
\[
\mu = \frac{1}{2}, \quad  \nu = 1, \quad \varsigma = - \frac{1}{2}.
\]
Therefore,
\[
\rho^*(n) = 2n + 1.
\]

Direct calculation of the condition number yields
\[
1 \leq \kappa^{\mathrm{Ch1}}(n, 2n + 1) < 1.08, \quad  n \geq 0.
\]
This bound for condition number shows that the Chebyshev
coefficients can be accurately computed without loss of accuracy if
the optimal radius is used.

\end{example}

\begin{example}
Consider the cosine function $f(x) = \cos(c x + d) $ with real
constants $c$, $d$ and $c > 0$. The exact Chebyshev coefficients are
\[
a_n = 2 \cos \left( d + n \frac{\pi}{2} \right) J_n(c), \quad n\geq
0,
\]
where $J_n(x)$ denotes the Bessel function of the first kind. We
have
\begin{align*}
M(\rho) & = \frac{1}{\pi} \int_{0}^{2 \pi} \left| \cos \left(
\frac{c}{2} (\rho e^{i\theta} + (\rho e^{i\theta})^{-1} ) + d\right)
\right| d\theta \\
& = \frac{1}{2 \pi} \int_{0}^{2\pi} \sqrt{ e^{c (\rho - \rho^{-1})
\sin\theta } + e^{- c (\rho - \rho^{-1}) \sin\theta } + 2 \cos( c(
\rho  + \rho^{-1} ) \cos\theta + 2d ) } d\theta  \\
& = \frac{1}{\pi} \int_{0}^{\pi} \sqrt{ e^{c (\rho - \rho^{-1})
\sin\theta } + e^{- c (\rho - \rho^{-1}) \sin\theta } + 2 \cos( c(
\rho  + \rho^{-1} ) \cos\theta + 2d ) } d\theta.
\end{align*}
For large $\rho$, noting that the sum in the last equality is
dominated by the first term, we have
\begin{align*}
M(\rho) & \sim \frac{1}{\pi} \int_{0}^{\pi}  e^{\frac{c}{2}
(\rho - \rho^{-1}) \sin\theta }  d\theta \\
& = I_0 \left( \frac{c}{2} (\rho - \rho^{-1}) \right) +
\frac{4}{\pi} \sum_{k = 0}^{\infty} \frac{(-1)^k }{ 2k + 1 } I_{2k +
1 } \left( \frac{c}{2} (\rho - \rho^{-1}) \right),
\end{align*}
where we have made use of the expansion \cite[Eqn. 9.6.35]{abramowitz1965handbook}.
Using the first term of the asymptotic expansion of $I_n(x)$, we
obtain
\[
M(\rho) \sim 2 \frac{ e^{ \frac{c}{2} \rho } }{ \sqrt{ c \pi \rho }
}, \quad \rho \rightarrow\infty.
\]
Identifying with Theorem \ref{thm:optimal radius for entire functions} leads to
\[
\mu = \frac{c}{2}, \quad  \nu = 1, \quad \varsigma = - \frac{1}{2}.
\]
Thus, we can derive the optimal radius for the cosine function
\[
\rho^*(n) = \frac{2n + 1}{c}.
\]
For example, for $c = 2$ and $d = 2$, direct calculation shows
\[
1 < \kappa^{\mathrm{Ch1}}(n, \frac{2n + 1}{c}) < 2.48, \quad  n \geq
1.
\]
\end{example}

\begin{example}
Consider a model function with a simple pole on the real line
\[
f(x) = \frac{1}{x - a},
\]
where $a>1$. The exact Chebyshev coefficients are given by
\cite[Eqn.~(5.14)]{mason2003chebyshev}
\[
a_n = -  \frac{2}{\sqrt{a^2 - 1}} (a - \sqrt{a^2 - 1} )^n , \quad
n\geq 0.
\]
Note that $f$ has a pole at $z = a$, we can deduce immediately that
$1 < \rho < A$ and $A = a + \sqrt{a^2-1}$. Direct calculation gives
\begin{align*}
M(\rho) & = \frac{1}{\pi} \int_{0}^{2 \pi} \frac{1}{ | \tfrac{1}{2}( \rho e^{i \theta} + (\rho e^{i\theta})^{-1} ) - a | }  d\theta \\
& = \frac{2 \rho}{\pi} \int_{0}^{2 \pi} \frac{1}{ | ( \rho e^{i
\theta}
)^2 - 2 a (\rho e^{i\theta}) + 1 | }  d\theta \\
& = \frac{2 \rho}{\pi} \int_{0}^{2 \pi} \frac{1}{ | ( \rho e^{i
\theta} -
A) (\rho e^{i\theta} - A^{-1} ) |}  d\theta. 
\end{align*}
The latter integral can be evaluated exactly in terms of elliptic integrals. An asymptotic expression for $\rho$ tending to $A$ is
\[
 M(\rho) \sim  \frac{4 \rho}{\pi} \frac{3 \log 2 + \log( A (A^2-1)) - \log (A^2+1) - \log (A-\rho)}{A^2-1}.
\]
Optimizing the condition number for large $n$ leads, after further asymptotic approximations, to
\[
 \rho^*(n) = A \left( 1 - \frac{1}{ n( 3 \, \log 2 + \log n )} \right).
\]
For small values of $A$ and $n$, a slightly more accurate
expression is
\[
 \rho^*(n) = A \left( 1 - \frac{1}{n ( 3 \, \log 2 - \log (A^2+1) + \log (A^2-1) + \log n) } \right).
\]
This leads for both expressions to a logarithmic growth of the
condition number as a function of $n$, approximately $\log n/\pi$.
Similar growth was observed for the computation of high derivatives
of this function in \cite[Example
5.2]{bornemann2010highorderderivatives}.

For example, when $a = 2$ direct calculation shows
\[
1 < \kappa^{\mathrm{Ch1}}(n, \rho^{*}(n) ) < 4.72, \quad  0 \leq n
\leq 10000,
\]
if we choose the optimal radius
\begin{equation}\label{eq:rho_pole}
\rho^{*}(n) = \left\{
                    \begin{array}{ll}
                     A \left( 1 - \frac{1}{ n( 3\log2 + \log n )}\right) ,  & \hbox{if $n \geq 1$,} \\
                     A \left( 1 - \frac{1}{3\log2}\right),  & \hbox{if $n = 0$.}
                    \end{array}
                  \right.
\end{equation}

The case where $f$ has a complex pole can be analyzed similarly, but is slightly more involved. Expression \eqref{eq:rho_pole} for the optimal radius continues to hold for a pole at the point $z_0$, with
\[
 A = | z_0 \pm \sqrt{z_0^2 - 1}|
\]
and where the sign is chosen such that $A > 1$. We omit the details of the derivation.

\begin{remark} In order to achieve the relative error tolerance $\epsilon$ by using the optimal radius, numerical experiments suggest
that we need about
\begin{align}\label{eq:estimate of nodes}
m_{\epsilon} \approx n ( 3\log2 + \log n) \log \epsilon^{-1}
\end{align}
nodes for large $n$. For example, we consider the computation of
$a_{100}$ of the function $f(x) = \frac{1}{x-4}$. To achieve
relative error $\epsilon = 10^{-13}$, we need $m_{\epsilon} \approx
20009$ nodes. Numerical results show that the relative error is
$2.0\times 10^{-14}$ when $m=20010$.
\end{remark}
\end{example}

\begin{example}
Consider the function
\[
f(x) = (c - x)^{\phi} g(x),
\]
where $\phi > 0$ is not an integer and $c > 1$ and $g(x)$ is an
analytic function at $x=c$. In this example, $f(x)$ has a branch
point at $x = c$. Direct calculations show that the maximum value of
$\rho$ is $\rho_{\max} = c+\sqrt{c^2-1}$ and
\begin{align}
M(\rho)& = \frac{1}{\pi} \int_{0}^{2 \pi} \left| f \left(
\frac{1}{2}
(\rho e^{i\theta} + (\rho e^{i\theta})^{-1} ) \right) \right| d\theta \nonumber \\
& = \frac{1}{\pi} \int_{0}^{2 \pi} \left| \left( c - \frac{1}{2}
(\rho e^{i\theta} + (\rho e^{i\theta})^{-1} ) \right)^{\phi} g\left(
\frac{1}{2}
(\rho e^{i\theta} + (\rho e^{i\theta})^{-1} ) \right) \right| d\theta \nonumber \\
& = \frac{1}{\pi} \int_{0}^{2 \pi} \left[ \frac{1}{4} ( \rho^2 +
\rho^{-2} ) - c( \rho + \rho^{-1} ) \cos\theta + \frac{1}{2}
\cos(2\theta) + c^2 \right]^{\frac{\phi}{2}} \nonumber \\
&~~~~~~~~~~~~~~~~~~~~ \left|g\left( \frac{1}{2} (\rho e^{i\theta} +
(\rho e^{i\theta})^{-1} ) \right)\right| d\theta.
\end{align}
It is easy to see that the integral in the last equation is bounded
when $\rho = \rho_{\max}$. Applying Theorem \ref{thm:modulus of f}
we have
\[
\lim_{\rho \rightarrow \rho_{\max}} \kappa^{\mathrm{Ch1}}(n, \rho) =
\frac{ \lim_{\rho \rightarrow \rho_{\max} } M(\rho) }{ |a_n|
\rho_{\max}^n },
\]
and the limit is finite. Thus, from Theorem \ref{thm:optimal for
finite radius}, we deduce that the optimal radius is $\rho^{*}(n) =
\rho_{\max}$ for large $n$. Moreover, from
\cite[Eqn.~(37)]{elliott1964chebyshev} we know that the Chebyshev
coefficients of $f(x)$ have the following estimate
\[
a_n \simeq - \frac{ 2\sin(\phi\pi) (c^2-1)^{\frac{\phi}{2}} g(c)
\Gamma(\phi+1) }{\pi n^{\phi+1} \rho_{\max}^n }.
\]
Thus, we can estimate the growth of the optimal condition number
\[
\kappa^{\mathrm{Ch1}}(n, \rho^{*}(n)) = \frac{ M(\rho) }{ |a_n|
\rho_{\max}^{n} } \sim \mathcal{O}( n^{\phi+1}), \quad
n\rightarrow\infty.
\]
which shows the optimal condition number grows algebraically as $n
\rightarrow \infty$.

\end{example}

\subsection{Identifying Chebyshev coefficients with Taylor coefficients}

An alternative way to reuse the results of \cite{bornemann2010highorderderivatives} is to put the integral representation of the Chebyshev coefficients \eqref{eq:ChebyT_contour_integral} into the form of a Cauchy integral like \eqref{eq:taylorcoefficient}. We will show that this can be achieved by a conformal map. The main advantage is that theoretical results can be reused. However, this identification between integrals does not seem to lead to a new or improved numerical scheme.

Let us first show that the Chebyshev coefficients of an analytic function can be viewed as the Taylor coefficients of another analytic function. An explicit form of this function can be established in terms of a contour integral of $f(z)$.

\begin{theorem}
Suppose that $a_k$ are the Chebyshev coefficients of the first kind
of the function $f(z)$ which is analytic inside and on the ellipse
$\mathcal{E}_{\rho}$. Then they are the Taylor coefficients of the following function
\begin{align}
H(x) = \frac{1}{\pi i} \oint_{ \mathcal{C}_{\rho} }  \frac{ f(
\tfrac{1}{2} ( u + u^{-1} ) ) }{u - x} d u,
\end{align}
and $H(x)$ is analytic inside the circle $\mathcal{C}_{\rho}$.
\end{theorem}
\begin{proof}
Suppose $a_k$ are the Chebyshev coefficients of $f$ and meanwhile
the Taylor coefficients of another function $H(x)$, e.g.
\[
f(x) = \sum_{k = 0}^{\infty}{'} a_k T_k(x), \quad  H(x) = \sum_{k =
0}^{\infty} a_k x^k.
\]
In view of the contour integral expression of $a_k$, we have
\begin{align}
H(x) & = \sum_{k = 0}^{\infty} a_k x^k \nonumber \\
& = \frac{1}{\pi i} \oint_{ \mathcal{C}_{\rho} }  f(z)
\sum_{k=0}^{\infty} x^k u^{-k-1} d u \nonumber \\
& = \frac{1}{\pi i} \oint_{ \mathcal{C}_{\rho} }  \frac{ f( z ) }{u
- x}
d u   \nonumber  \\
& = \frac{1}{\pi i} \oint_{ \mathcal{C}_{\rho} }  \frac{ f(
\tfrac{1}{2} ( u + u^{-1} ) ) }{u - x} d u .
\end{align}
This completes the proof.
\end{proof}

\begin{corollary}
Suppose that $b_k$ are the Chebyshev coefficients of the second kind
of the function $f(z)$ which is analytic inside and on the ellipse
$\mathcal{E}_{\rho}$, then they are the Taylor coefficients of the
following function
\begin{align}
H(x) = \frac{1}{ 2 \pi i} \oint_{ \mathcal{C}_{\rho} }  \frac{ f(
\tfrac{1}{2} ( u + u^{-1} ) ) }{u - x} (1 - u^{-2} ) d u,
\end{align}
and $H(x)$ is analytic inside the circle $\mathcal{C}_{\rho}$.
\end{corollary}

In the following we present some concrete examples, where $H(x)$ can be deduced in (almost) closed form.

\begin{example}
Consider the exponential function $f(x) = e^x$. We have
\begin{align}
H(x) = \frac{1}{\pi i} \oint_{ \mathcal{C}_{\rho} }  \frac{ e^{
\tfrac{1}{2} ( u + u^{-1} ) } }{u - x} d u.
\end{align}
Direct calculations show that
\begin{align}
a_k & = \frac{H^{(k)}(0)}{k!} \nonumber \\
& = \frac{1}{\pi i } \oint_{ \mathcal{C}_{\rho} }  e^{ \tfrac{1}{2}
( u + u^{-1} ) }  \frac{1}{(u - x)^{k+1} } d u  \nonumber \\
& = \sum_{m = 0}^{\infty} \frac{2}{ 2^{k + 2m } \Gamma(k+m+1)
\Gamma(m+1)  } \nonumber \\
& = 2 I_k(1).
\end{align}
Thus, we have
\[
H(x) = 2 \sum_{k = 0}^{\infty} I_k(1) x^k,
\]
which is an entire function.
\end{example}

\begin{example}
Consider the function
\[
f(x) = \frac{1}{x - a}, \quad a>1.
\]
Using the residue theorem, we obtain
\begin{align}
H(x) & =  \frac{1}{\pi i} \oint_{ \mathcal{C}_{\rho} }  \frac{ f(
\tfrac{1}{2} ( u + u^{-1} ) ) }{u - x} d u   \nonumber \\
& = \frac{1}{\pi i} \oint_{ \mathcal{C}_{\rho} } \frac{ 2 u }{ u^2 - 2au + 1 } \frac{ 1 }{u - x} d u   \nonumber \\
& = \frac{ 2 (a + \sqrt{a^2 - 1} ) }{ ( x - ( a + \sqrt{a^2 - 1} ) )
\sqrt{a^2 - 1} },
\end{align}
and $H(x)$ is analytic inside the circle $ |z| < a + \sqrt{a^2 -
1}$.
\end{example}


\section{Two strategies for computing the Chebyshev coefficients }\label{sec:two strategies}
In this section we present two strategies for computing the first
$N+1$ Chebyshev coefficients of analytic functions. The first
strategy maximizes the computational efficiency and can be performed
via the FFT. The second strategy minimizes the loss of accuracy for
each coefficient and is stable with respect to relative errors.

\subsection{Fast algorithms to maximize the efficiency
}\label{sec:maximize the efficiency}

Note that the sum \eqref{eq:trap for cheb1} for the computation of
$\{a_k\}_{k=0}^{N}$ is suitable for using FFT if $\rho$ is fixed for
all $\{ a_k(m,\rho) \}_{k=0}^{N}$. Therefore, by choosing the same
value of $\rho$ for each expansion coefficient, either for integral
\eqref{eq:trap for cheb1} or integral \eqref{eq:trap for cheb2}, the
Chebyshev coefficients $a_n$ and $b_n$ can be computed efficiently
with a single FFT and this process can be performed in
$\mathcal{O}(m\log m)$ operations. In the following we present some
numerical experiments to show the performance of the FFT algorithm.
\begin{example}

First, we consider the transcendental function $f(x) = e^x$.
Clearly, this function is entire and thus $1 \leq \rho < \infty$. In
Figure \ref{fig:Absolute errors of Chebyshev coefficients one} we
show the absolute and relative errors of the FFT algorithm for
computing the first $N+1$ Chebyshev coefficients. We see that the
absolute errors are uniformly small for $0\leq k \leq N$ when we
choose $\rho = 1$. When $\rho>1$, we see that the absolute errors
decrease exponentially as $k$ increases. However, we also observe
that the absolute errors deteriorate for the first several Chebyshev
coefficients if $\rho=40$. As for the relative error, we observe
that it has the fastest rate of exponential growth when $\rho=1$ and
then becomes better as $\rho$ increases. When $\rho=40$, we see that
the relative error deteriorates for the first several Chebyshev
coefficients.
\end{example}

\begin{figure}[h]
\centering
\begin{overpic}
[width=7.1cm,height=6cm]{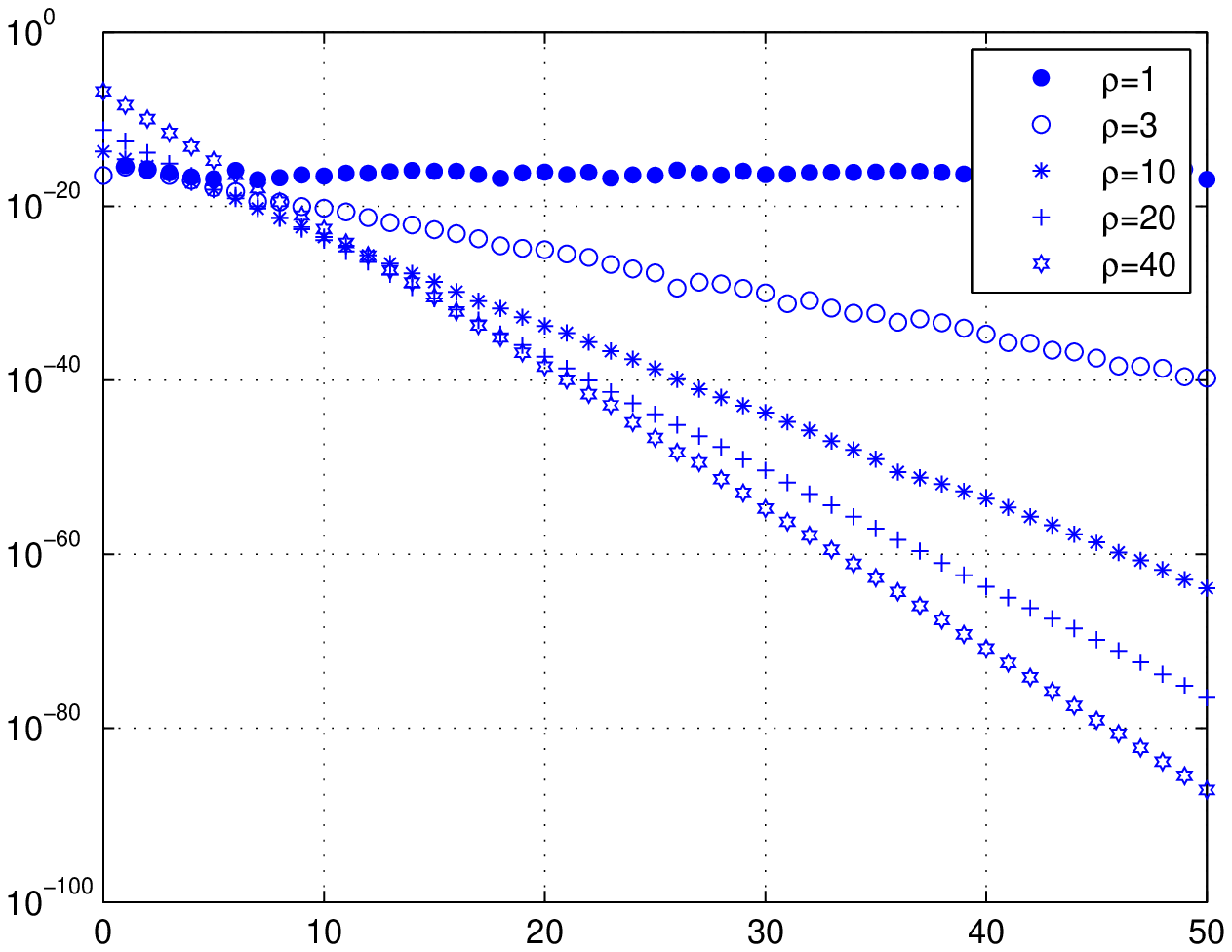}
\end{overpic}
~~
\begin{overpic}
[width=7.1cm,height=6cm]{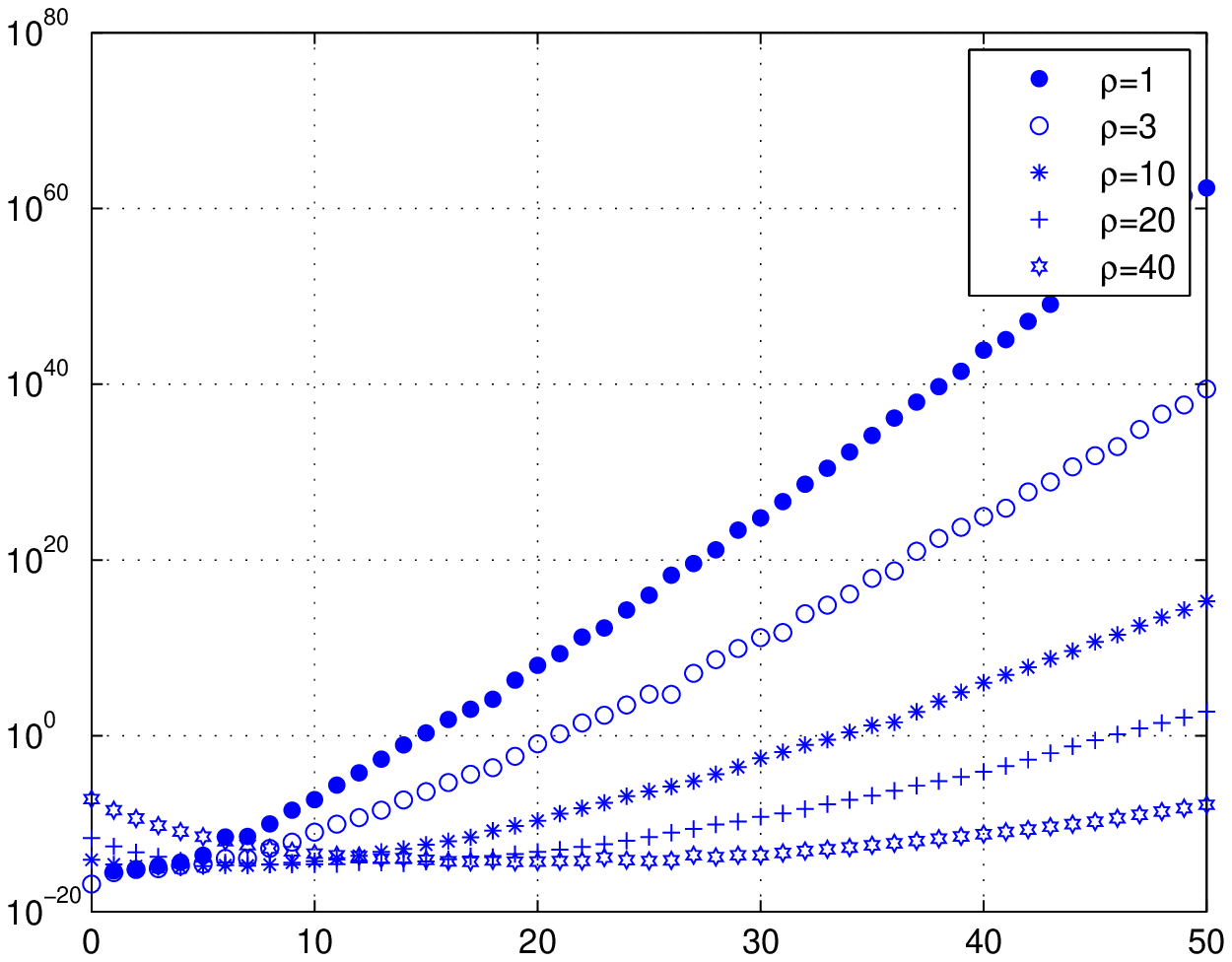}
\end{overpic}
\caption{Absolute errors (left) and relative errors (right) of the
computed Chebyshev coefficients $\{ a_k(m,\rho) \}_{k=0}^{N}$ for
the function $f(x) = e^x$. Here we choose $N=50$ and $m=2N+1$. }
\label{fig:Absolute errors of Chebyshev coefficients one}
\end{figure}

\begin{example}
We consider the function $f(x) = \frac{1}{x-2}$. Note that this
function has a real pole at $x=2$ and we can deduce that $1 \leq
\rho < 2+\sqrt{3} \approx 3.732$. In our computations we choose $m =
4N+2$ and we have tested several values of $\rho$. Numerical results
are presented in Figure \ref{fig:Absolute errors of Chebyshev
coefficients two}. We see that, similar to the above example, the
absolute errors are also uniformly small when we choose $\rho = 1$
and decrease exponentially as $k$ increases when $\rho>1$. As for
the relative error, we observe that it grows exponentially with the
fastest rate when $\rho=1$ and then becomes better as $\rho$ grows.
In particular, the relative error is less than $10^{-11}$ for all
$\{a_k\}_{k=0}^{N}$ when $\rho = 3$. We point out that the absolute
and relative errors will deteriorate simultaneously when $\rho$ is
very close to its maximum value. This is due to the fact that the
term $\| f - s_{m} \|_{ \mathcal{E}_{\rho}}$ in Theorem
\ref{thm:absolute stability} tends to infinity when $m$ is fixed and
$\rho$ tends to its maximum value.
\end{example}

\begin{figure}[h]
\centering
\begin{overpic}
[width=7.1cm,height=6cm]{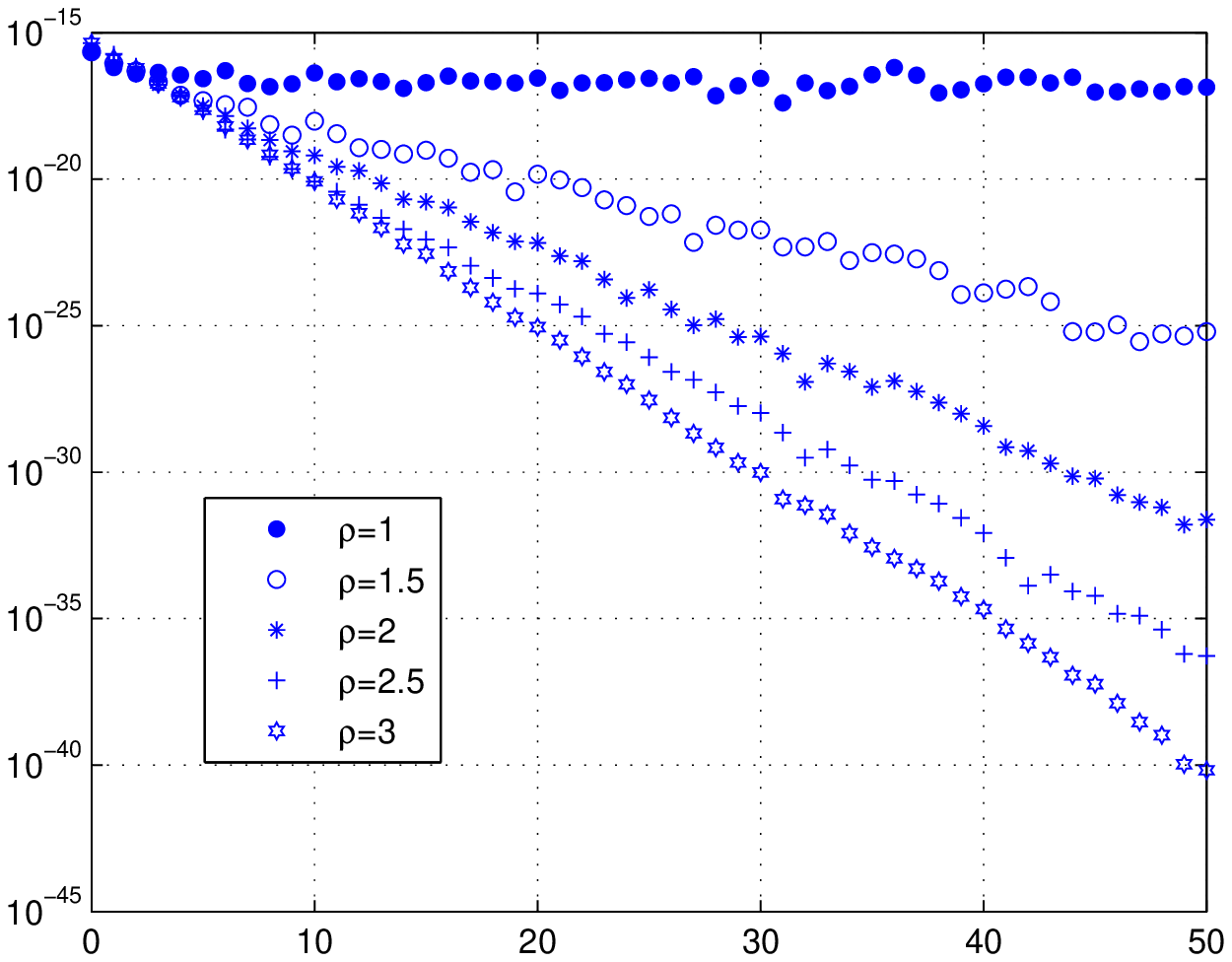}
\end{overpic}
~~
\begin{overpic}
[width=7.1cm,height=6cm]{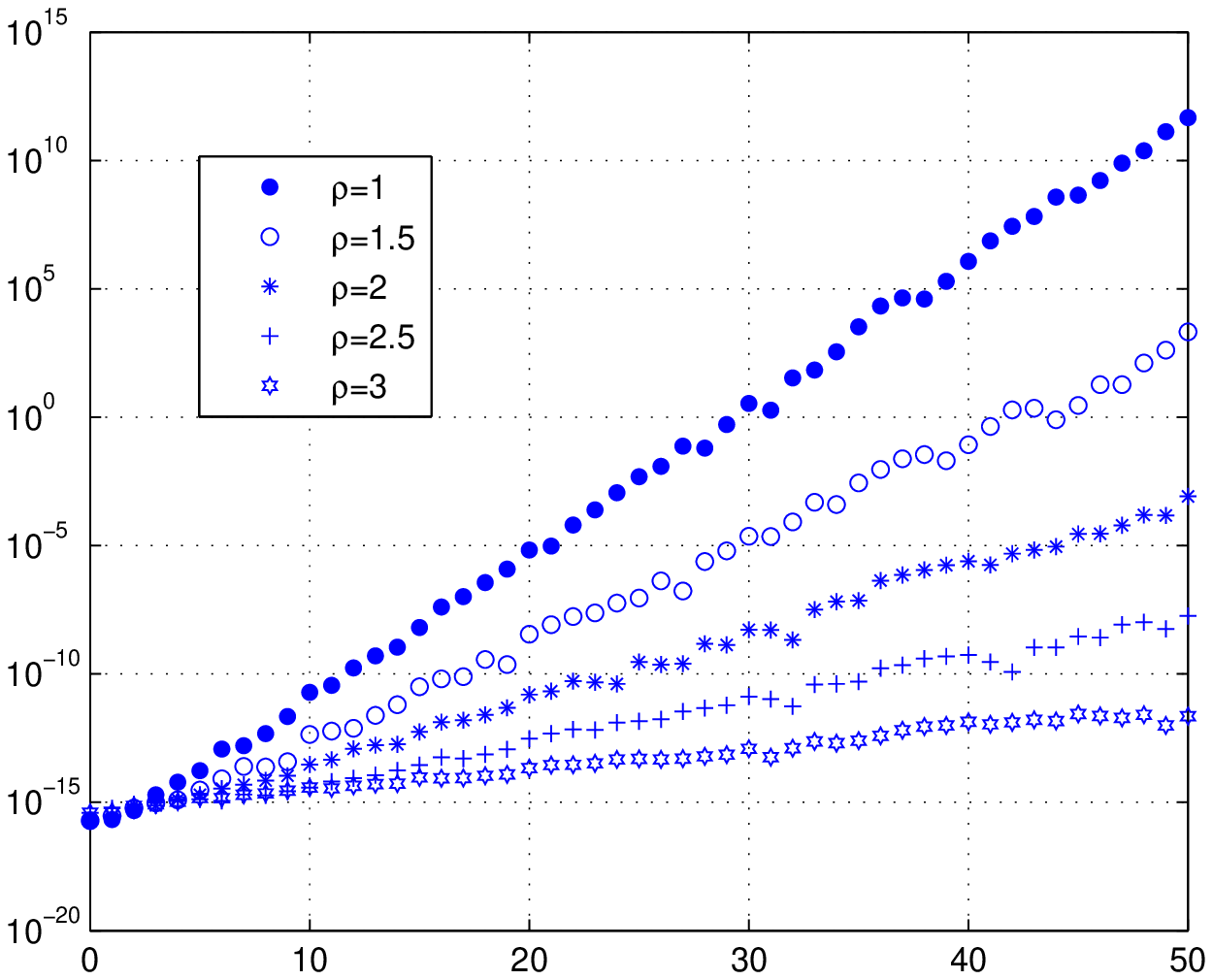}
\end{overpic}
\caption{Absolute errors (left) and relative errors (right) of the
computed Chebyshev coefficients $\{ a_k(m,\rho) \}_{k=0}^{N}$ for
the function $f(x) = \frac{1}{x-2}$. Here we choose $N=50$ and
$m=4N+2$. } \label{fig:Absolute errors of Chebyshev coefficients
two}
\end{figure}

Finally, we conclude this subsection with several remarks.
\begin{remark}
For transcendental functions, the computation of their Chebyshev
coefficients by a single $\rho$ may suffer from instability when
$\rho\gg1$.
\end{remark}

\begin{remark}
Numerical experiments show that, for a fixed $\rho$, it is
sufficient to choose $m=\mathcal{O}(N)$ such that the absolute
errors of the first $N+1$ Chebyshev coefficients are less than a
given tolerance uniformly. Thus the cost is $\mathcal{O}(N\log N)$
operations for computing the first $N+1$ Chebyshev expansion
coefficients.
\end{remark}

\begin{remark}
If we are concerned only with the absolute errors of Chebyshev
coefficients, it is sufficient to choose $\rho=1$ and $m =
\mathcal{O}(N)$. This leads to a fast algorithm which costs only
$\mathcal{O}(N\log N)$ operations for computing the first $N+1$
Chebyshev coefficients. However, if we are concerned with the
relative errors, the situation will change completely and it is
dangerous to choose $\rho=1$ since they have the fastest rate of
exponential growth.
\end{remark}

\begin{remark}
If $f(x)$ is analytic only in a neighborhood of $[-1,1]$, it is
possible to compute all Chebyshev coefficients $\{a_k\}_{k=0}^{N}$
by choosing a single $\rho$ such that their relative errors are less
than a given tolerance.
\end{remark}

\subsection{Maximizing the accuracy of Chebyshev coefficients}\label{sec:accuracy}
We can see from the above subsection that the relative errors of
Chebyshev coefficients may grow exponentially as $k$ grows if we
compute them by using the same $\rho$. To remedy this drawback, we
propose an alternative strategy and compute each Chebyshev
coefficient $a_k$ by using its optimal $\rho^{*}(k)$. This leads to
an accurate algorithm which minimizes the loss of accuracy with
respect to relative errors.

In Figure \ref{fig:Relative errors of Chebyshev coefficients} we
show relative errors of this strategy for computing the first $N+1$
Chebyshev coefficients of the functions $f(x) = e^x,~\frac{1}{x-2}$.
For the former function, each Chebyshev coefficient $a_k$ is
computed by \eqref{eq:trap for cheb1} with $\rho = 2k+1$ and
$m=2N+1$. For the latter function, each Chebyshev coefficient $a_k$
is evaluated by the trapezoidal rule \eqref{eq:trap for cheb1} with
the optimal radius \eqref{eq:rho_pole} and the number of points in
the trapezoidal rule is chosen as
\begin{align*}
m = \max\{ k( 3\log2 + \log k) \log\epsilon^{-1}, 50 \}, \quad 0
\leq k \leq 100,
\end{align*}
and we choose $\epsilon = 10 ^{-14}$. We can see that the Chebyshev
coefficients can be evaluated very accurately with respect to
relative errors.

\begin{figure}[h]
\centering
\begin{overpic}
[width=7.1cm,height=6cm]{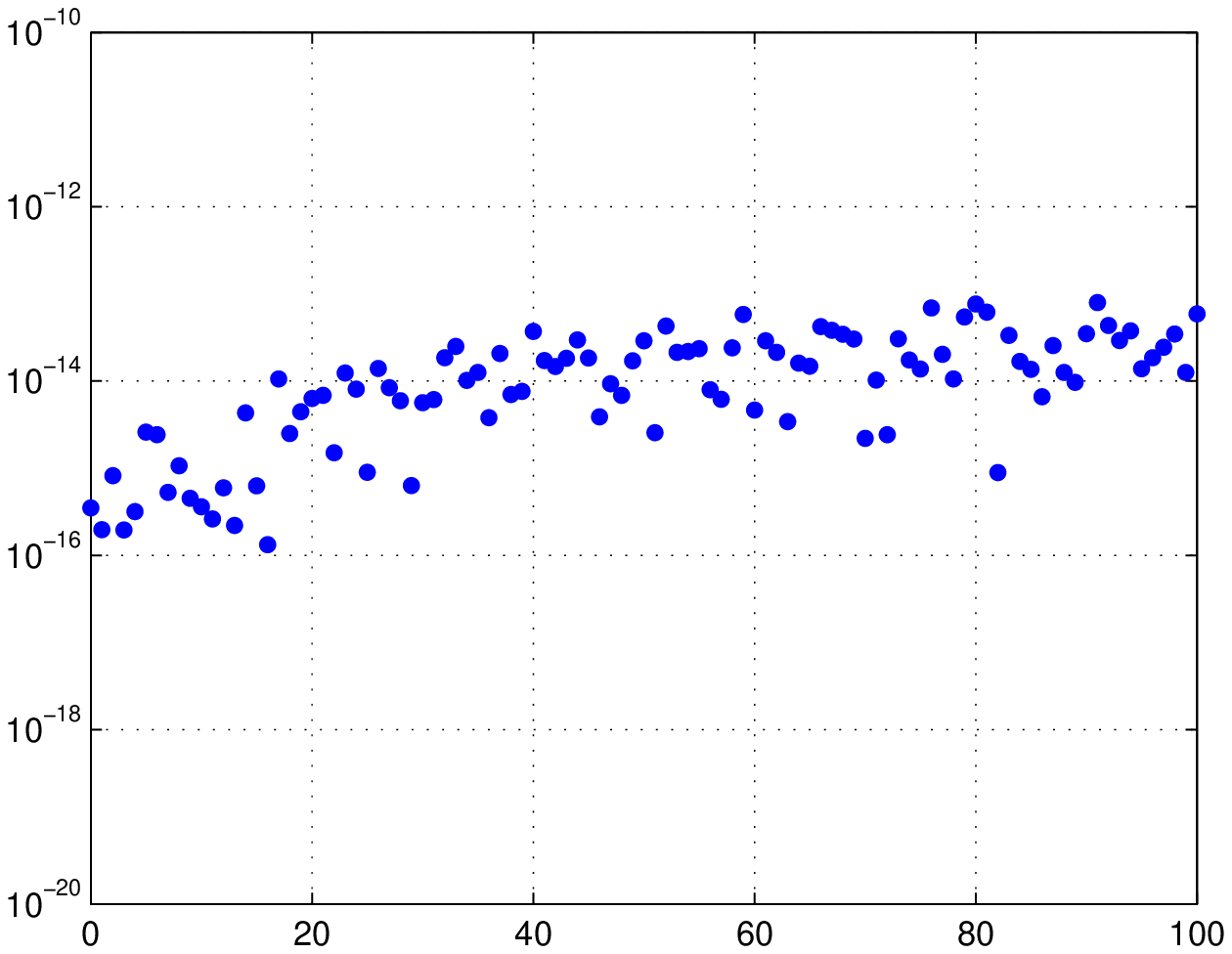}
\end{overpic}
~~
\begin{overpic}
[width=7.1cm,height=6cm]{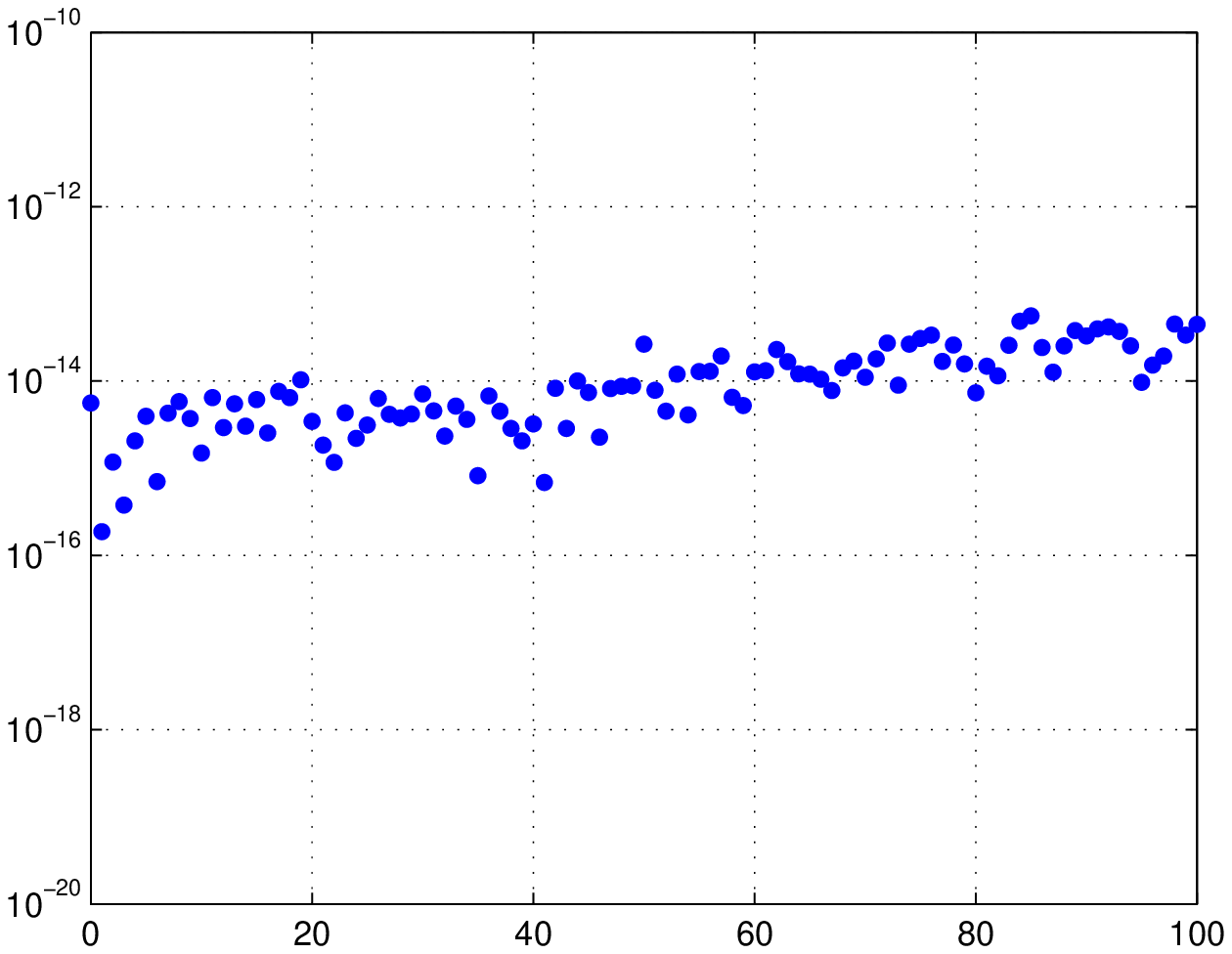}
\end{overpic}
\caption{Relative errors of the computed Chebyshev coefficients $\{
a_k(m,\rho^{*}(k)) \}_{k=0}^{N}$ of $f(x) = e^x$ (left) and $f(x) =
\frac{1}{x-2}$ (right). Here $N=100$. } \label{fig:Relative errors
of Chebyshev coefficients}
\end{figure}

Bornemann analyzes the number of quadrature points $m$ to use for the computation of the Taylor coefficient $a_n$ in terms of $n$, and this depends on the nature of the function, in particular its analyticity properties \cite[\S2]{bornemann2010highorderderivatives}. We found experimentally that these results can be reused in the setting of the computation of Chebyshev coefficients, and this has guided the choice of $m$ for the examples in the current paper.

\section{Chebyshev spectral differentiation}\label{sec:differentiation}
In this section we show some examples to illustrate the accuracy of
Chebyshev spectral differentiation based on the spectral expansions. 
Let
\[
f_N^{C}(x) = \sum_{k=0}^{N}{'} a_k T_k(x)
\]
denote the truncated Chebyshev expansion. Then the derivatives of
$f(x)$ can be approximated by the corresponding derivatives of
$f_N^{C}(x)$, e.g.
\[
f^{(s)}(x) \approx \frac{d^s}{dx^s}f_N^{C}(x).
\]
Let
\[
\frac{d^s}{dx^s}f_N^{C}(x) = \sum_{k=0}^{N}{'} a_k^{(s)} T_k(x).
\]
Then the coefficients $a_k^{(s)}$ can be evaluated by using the
following recurrence relation \cite[p.~498]{boyd2001spectral}
\begin{align}\label{eq:recurrence relation}
a_{k-1}^{(s)} = a_{k+1}^{(s)} + 2ka_{k}^{(s-1)}, \quad k =
N-s+1,\ldots,1,
\end{align}
where $a_{N-s+2}^{(s)} = a_{N-s+1}^{(s)} = 0$. 
Moreover, the initial coefficients are given by $a_k^{(0)} = a_k$
for $0 \leq k \leq N$.

\begin{example} We consider the accuracy of the Chebyshev spectral differentiation for the test function $f(x) =
e^x$. Each Chebyshev coefficient $a_k$ is evaluated by the
trapezoidal rule \eqref{eq:trap for cheb1} with the optimal radius
and the number of points in the trapezoidal rule is $m = 100$. In
Figure \ref{fig:accuracy of Chebyshev spectral diff for exponential
function} we present the pointwise errors in the evaluation of the
$s$-th order derivative of $f(x)$ by the truncated Chebyshev
spectral expansion $f_N^{C}(x)$. The error is measured at $100$
equispaced points in $[-1,1]$. As can be seen, the error of the
Chebyshev spectral differentiation is always very close to machine
precision.
\end{example}

\begin{figure}[h]
\centering
\begin{overpic}
[width=4.5cm]{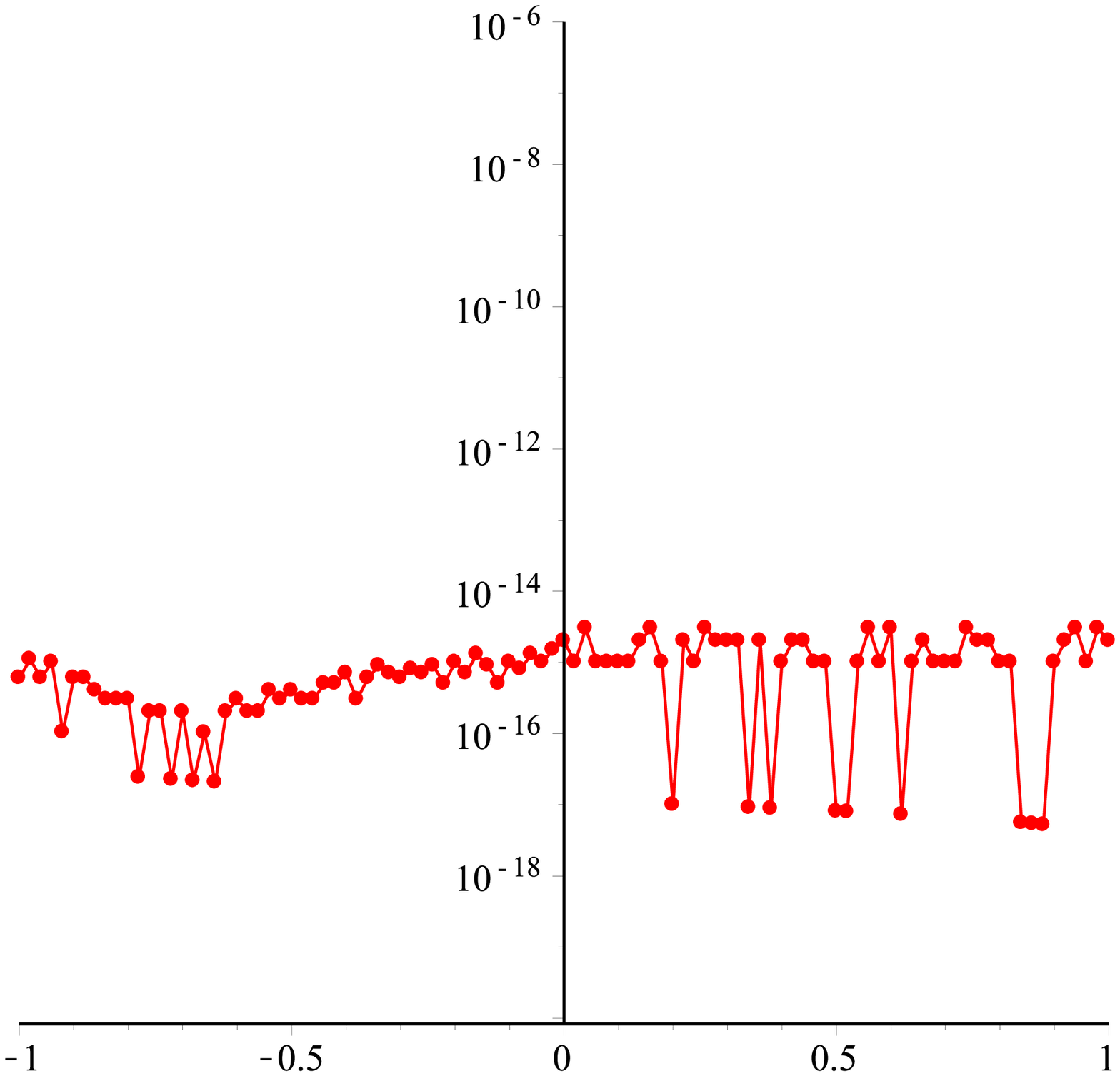}
\end{overpic}
\quad
\begin{overpic}
[width=4.5cm]{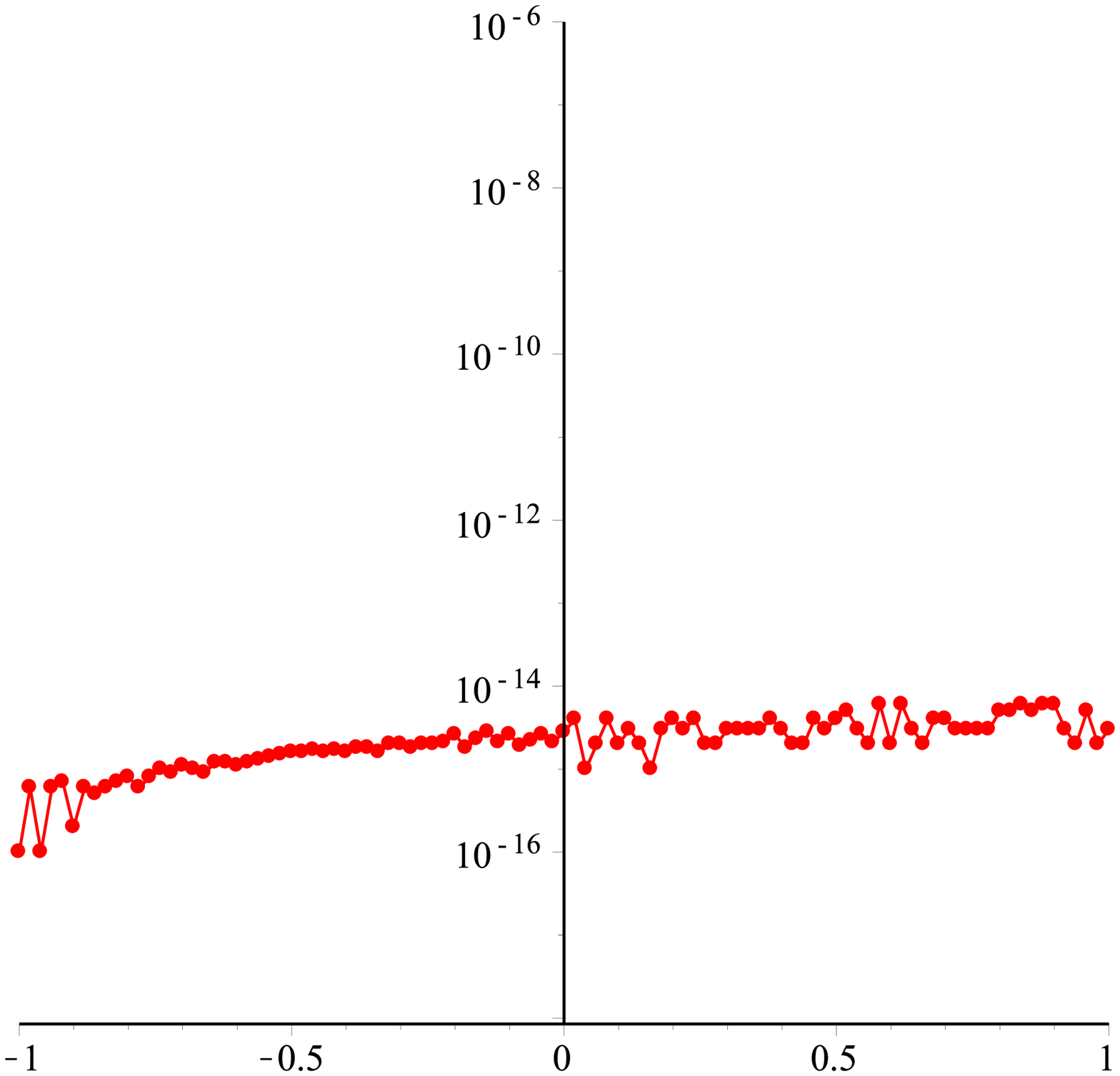}
\end{overpic}
\quad
\begin{overpic}
[width=4.5cm]{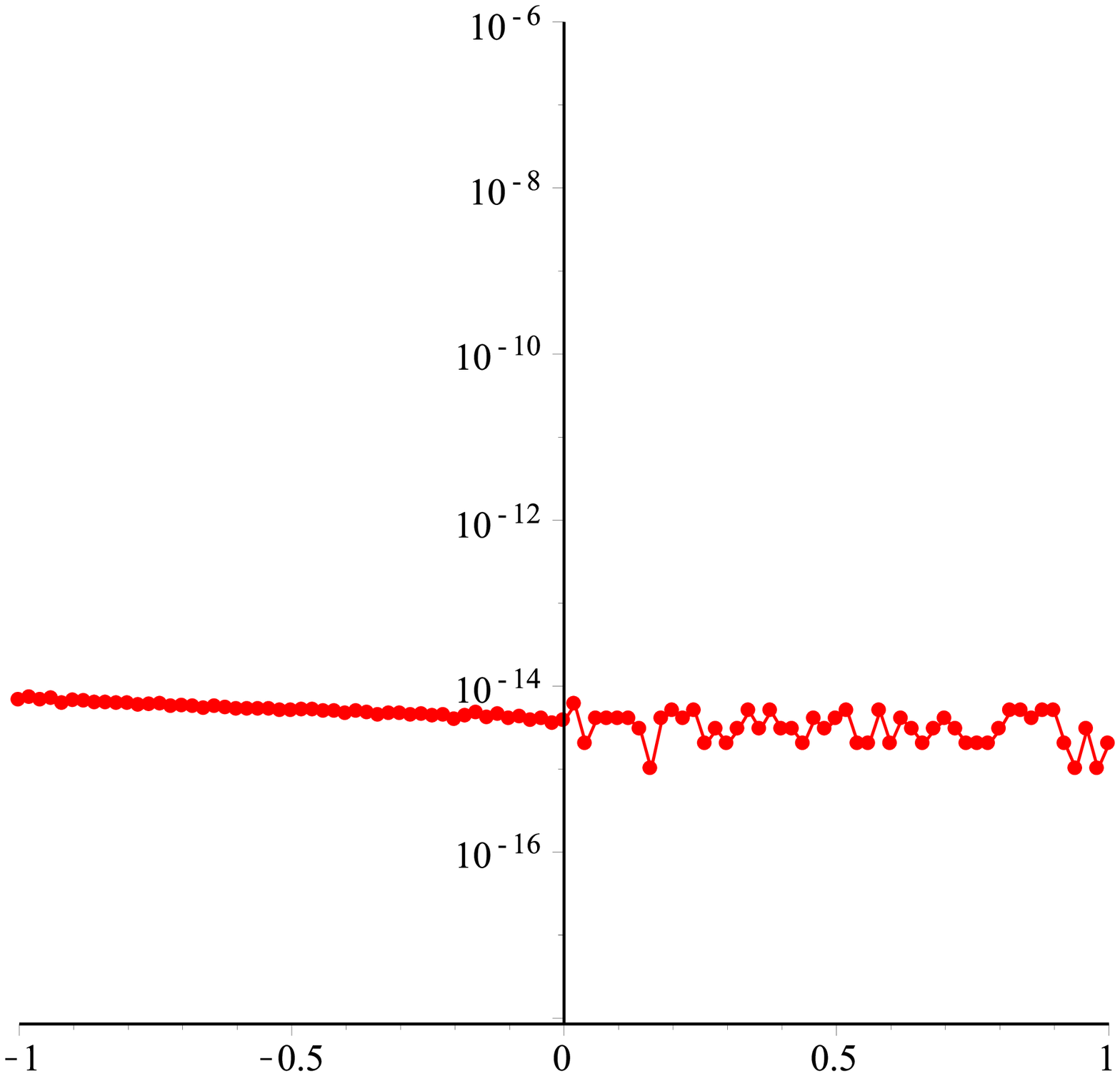}
\end{overpic}
\caption{Errors of the $s$-th order derivative of the truncated
Chebyshev expansion $f_N^{C}(x)$. Here we choose $N=100$ and $s=5$
(left), $s=20$ (middle) and $s=80$ (right).  } \label{fig:accuracy
of Chebyshev spectral diff for exponential function}
\end{figure}

\begin{example} We consider the accuracy of the Chebyshev spectral differentiation for the function $f(x) =
\cos (x)$. Each Chebyshev coefficient $a_k$ is evaluated by the
trapezoidal rule \eqref{eq:trap for cheb1} with the optimal radius
and the number of points in the trapezoidal rule is $m = 100$. In
Figure \ref{fig:accuracy of Chebyshev spectral diff for cosine
function} we present the pointwise errors in the evaluation of the
$s$-th order derivative of $f(x)$ by the truncated Chebyshev
spectral expansion $f_N^{C}(x)$.
\end{example}

\begin{figure}[h]
\centering
\begin{overpic}
[width=4.5cm]{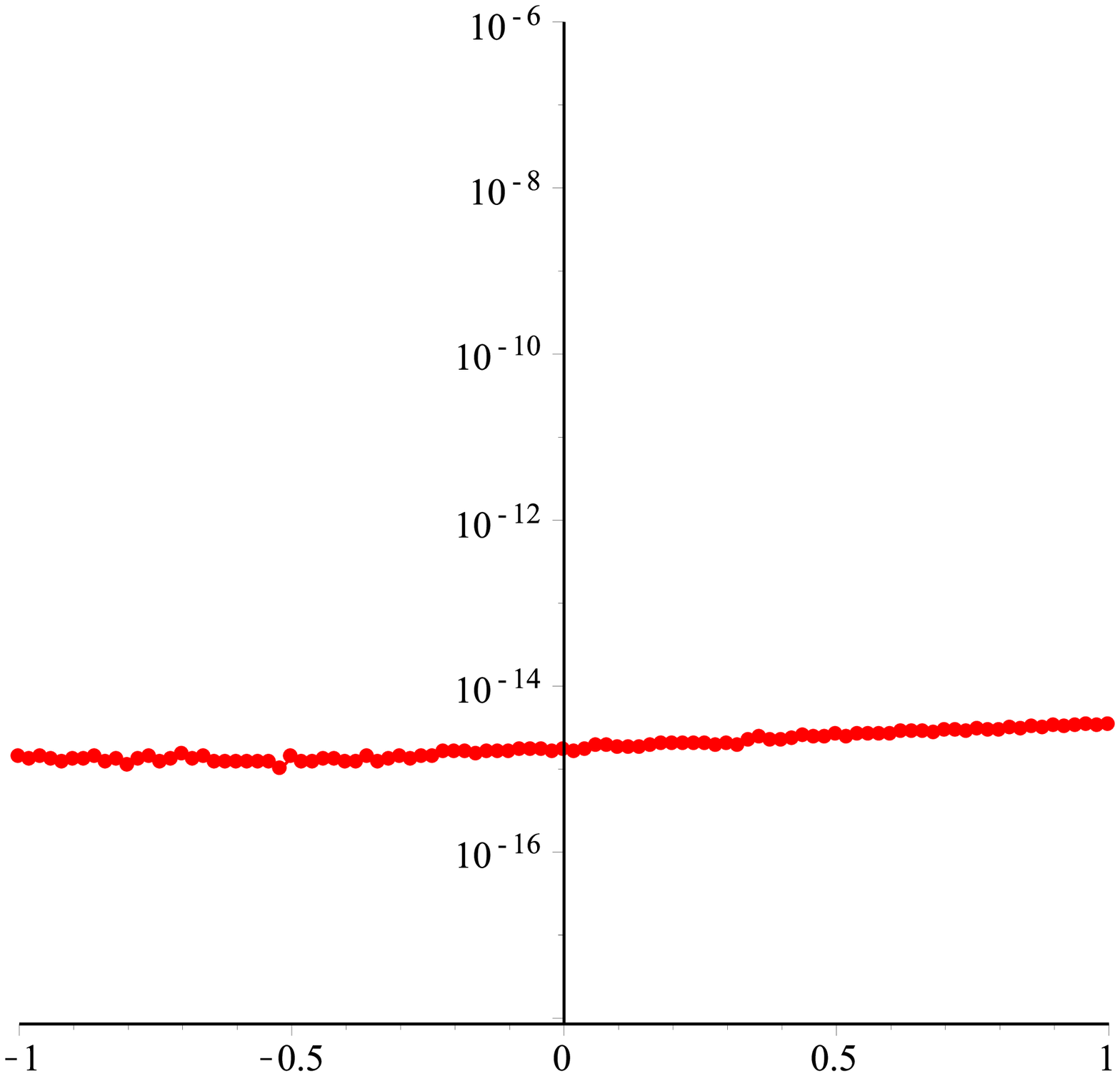}
\end{overpic}
\quad
\begin{overpic}
[width=4.5cm]{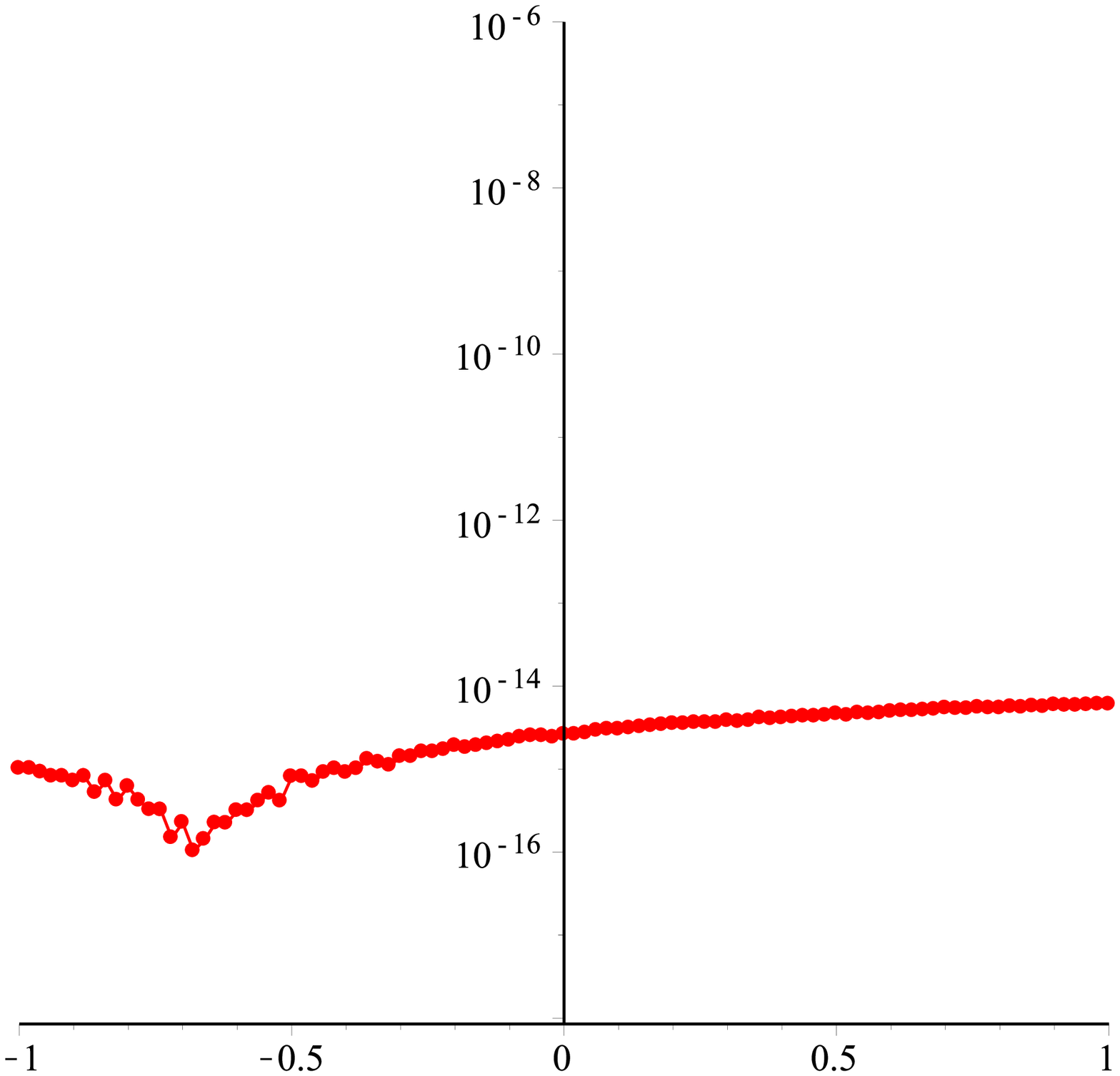}
\end{overpic}
\quad
\begin{overpic}
[width=4.5cm]{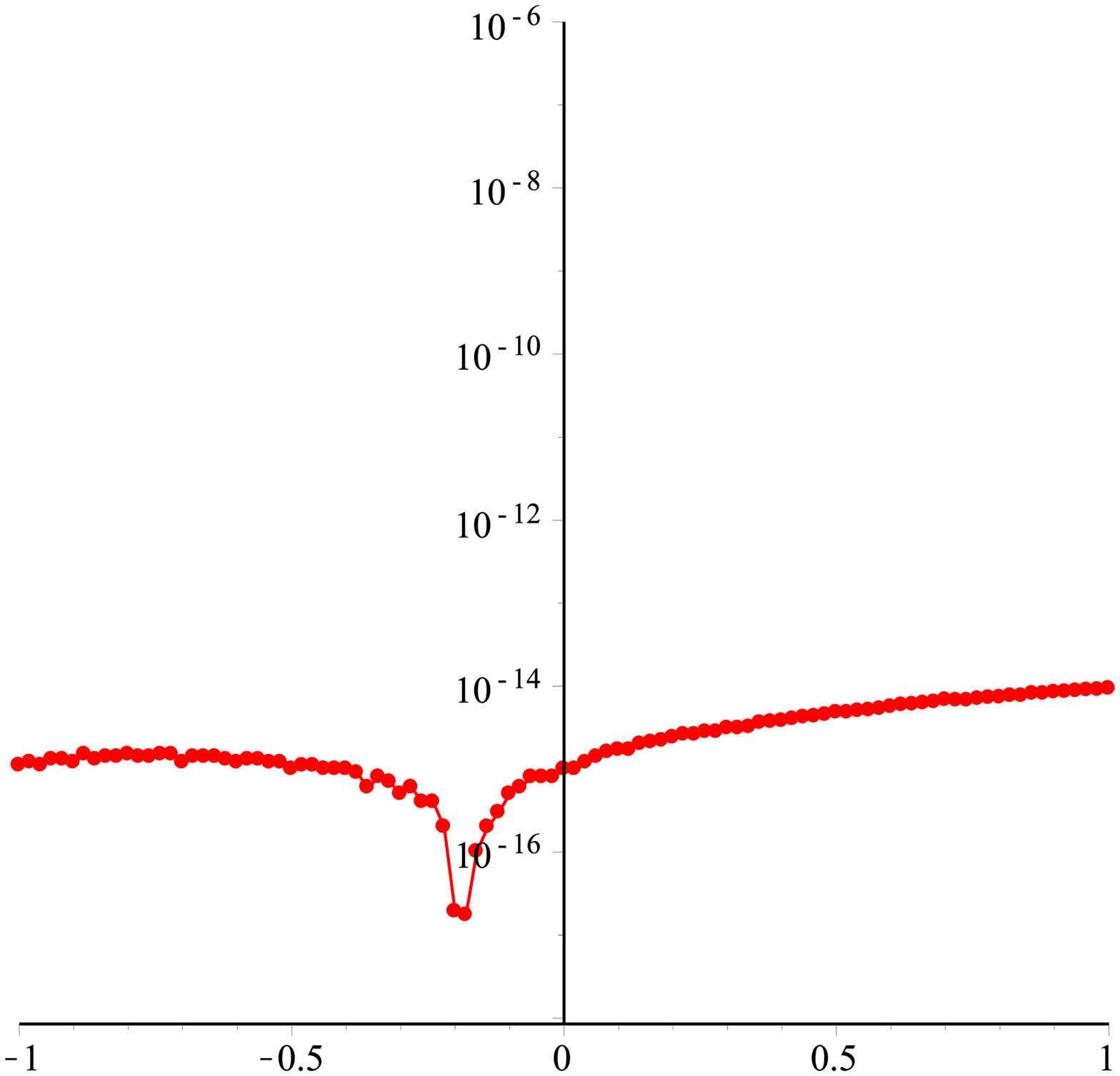}
\end{overpic}
\caption{Errors of the $s$-th order derivative of the truncated
Chebyshev expansion $f_N^{C}(x)$. Here we choose $N=100$ and $s=10$
(left), $s=40$ (middle) and $s=80$ (right).  } \label{fig:accuracy
of Chebyshev spectral diff for cosine function}
\end{figure}

\begin{example} Finally, we consider the accuracy of the Chebyshev spectral differentiation for the test function $f(x) =
\frac{x+1}{x^2+4} $. Each Chebyshev coefficient $a_k$ is evaluated
by the trapezoidal rule \eqref{eq:trap for cheb1} with the optimal
radius and the number of points in the trapezoidal rule is chosen as
\begin{align*}
m = \max\{ n( 3\log2 + \log n) \log\epsilon^{-1}, 50 \}, \quad 0
\leq n \leq 100,
\end{align*}
and we choose $\epsilon = 10 ^{-16}$. The pointwise error of the
Chebyshev spectral differentiation in the evaluation of the $s$-th
order derivative of $f(x)$ is displayed in Figure \ref{fig:accuracy
of Chebyshev spectral diff for pole function}.
\end{example}

\begin{figure}[h]
\centering
\begin{overpic}
[width=4.5cm]{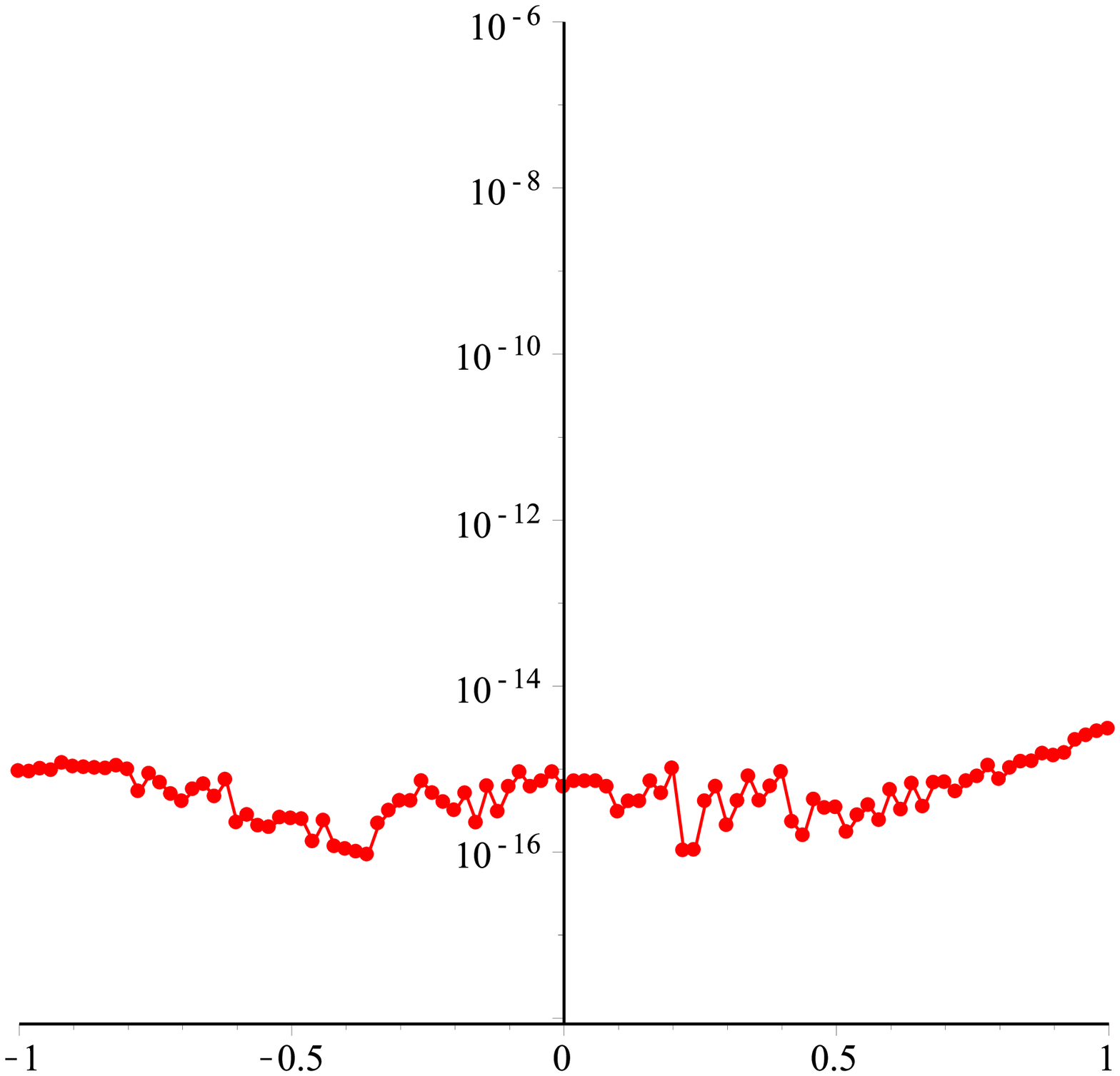}
\end{overpic}
\quad
\begin{overpic}
[width=4.5cm]{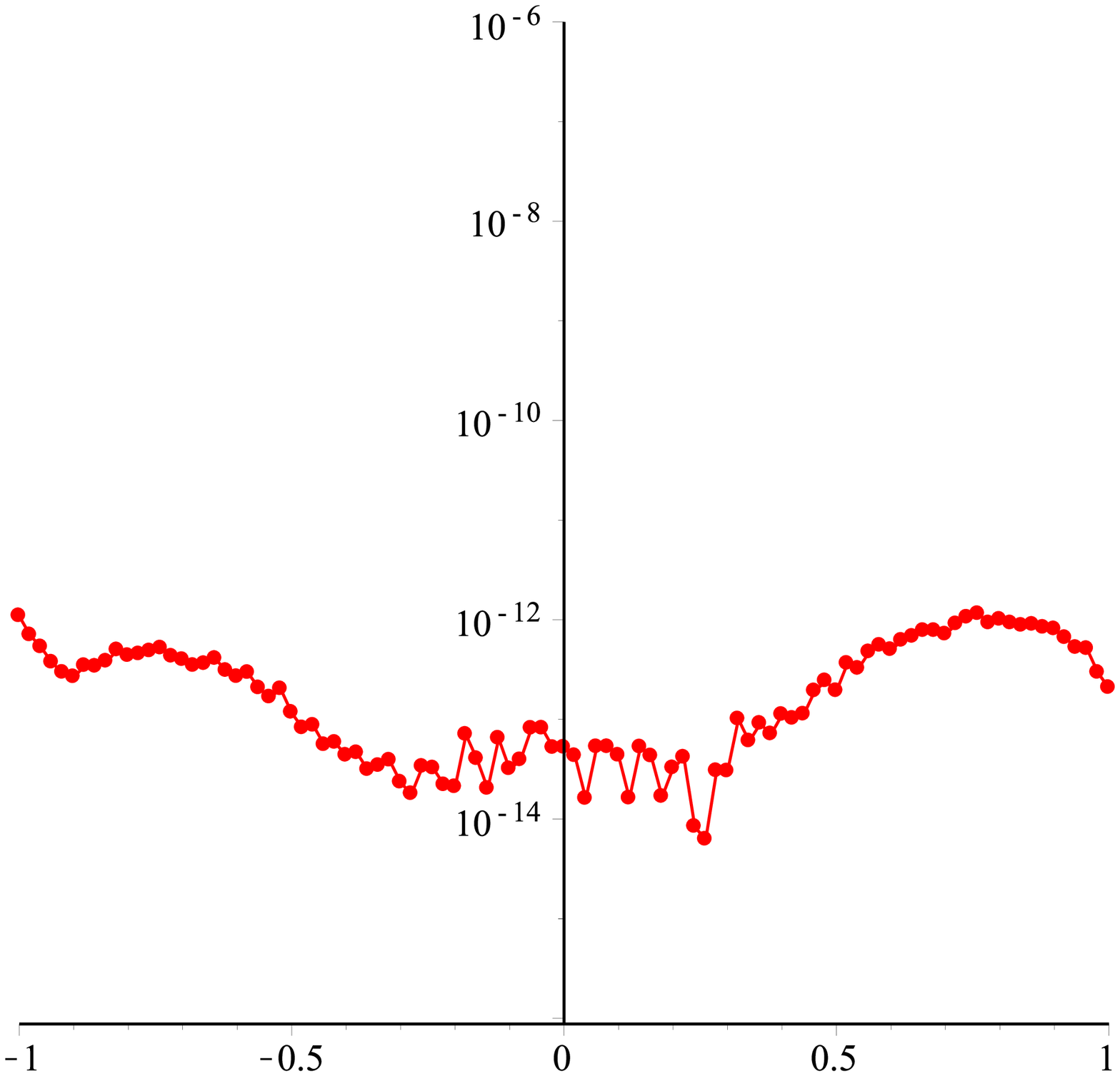}
\end{overpic}
\quad
\begin{overpic}
[width=4.5cm]{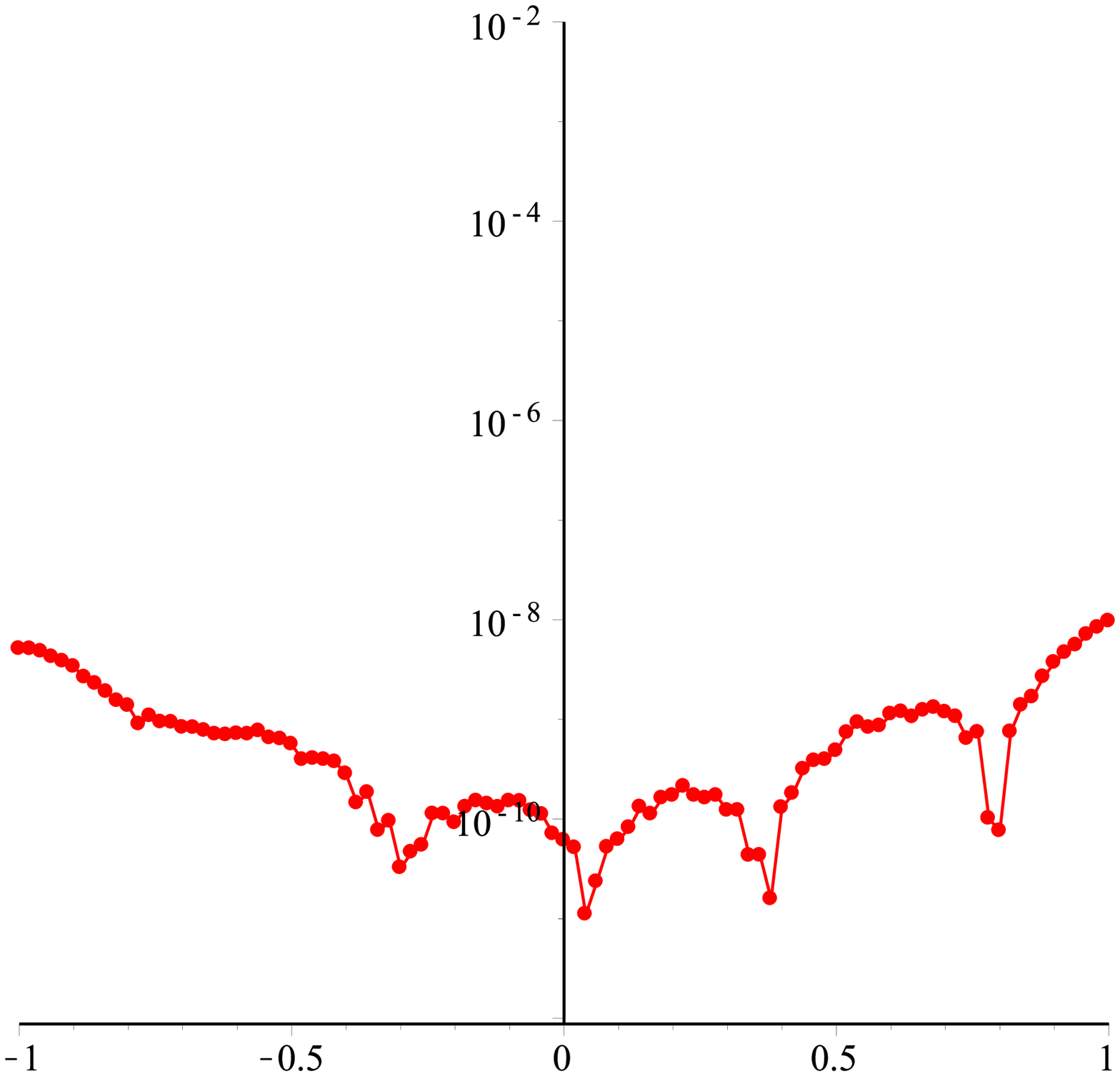}
\end{overpic}
\caption{Errors of the $s$-th order derivative of the truncated
Chebyshev expansion $f_N^{C}(x)$. Here we choose $N=100$ and $s=4$
(left), $s=8$ (middle) and $s=12$ (right).  } \label{fig:accuracy of
Chebyshev spectral diff for pole function}
\end{figure}

\section{Computing the roots of derivatives of analytic functions}\label{sec:roots}
One powerful application of the truncated Chebyshev expansion of an analytic function $f(x)$ is that it can be used to
compute the roots of $f(x)$ on the interval $[-1,1]$. The main idea
is that the roots of a Chebyshev series are the eigenvalues of a
colleague matrix whose elements are simple functions of the
coefficients of the Chebyshev series. For the sake of clarity, we
state it in the following.
\begin{theorem}\label{thm:colleague matrix}
The roots of the Chebyshev series
\[
p(x) = \sum_{k=0}^{n} c_k T_k(x), \quad c_n \neq 0,
\]
are the eigenvalues of the following colleague matrix
\begin{equation}
A = \begin{pmatrix}
  0            &  1            &               &                  &             &               \\
  \frac{1}{2}  &  0            &   \frac{1}{2} &                  &             &               \\
               & \frac{1}{2}   &  0            &   \frac{1}{2}    &             &               \\
               &               &  \ddots       &  \ddots          & \ddots      &               \\
               &               &               &                  &             &  \frac{1}{2}  \\
               &               &               &                  & \frac{1}{2} &  0            \\
 \end{pmatrix} -
 \frac{1}{2c_n}
 \begin{pmatrix}
               &            &         &         &             \\
               &            &         &         &             \\
               &            &         &         &             \\
               &            &         &         &              \\
               &            &         &         &              \\
           c_0 & c_1        & c_2     & \cdots  &   c_{n-1}   \\
 \end{pmatrix}.
\end{equation}
If there are multiple roots, these correspond to eigenvalues with
the same multiplicities.
\end{theorem}
\begin{proof}
See \cite[Thm.~18.1]{trefethen2012atap}.
\end{proof}

In practice, it is of particular interest to compute the roots of
derivatives of a smooth function. For example, the roots of the
first and second order derivatives of a function correspond exactly
to its maxima and inflexion points. In the following, we show the
performance of our methods applied to the computation of derivatives
of a transcendental function. As shown in the above section, the
Chebyshev coefficients of the $s$-th order derivatives of
$f_N^{C}(x)$ can be computed via the recurrence relation
\eqref{eq:recurrence relation} and thus the roots of
$\frac{d^s}{dx^s}f_N^{C}(x)$ can be computed by using Theorem
\ref{thm:colleague matrix}.

\begin{example}
Consider
\[
f(x) = e^{2x} + \cos(2x+3), \quad x\in[-1,1],
\]
For each Chebyshev coefficient $a_n$ of $f(x)$, it is not difficult
to deduce that the optimal radius is $\rho^{*}(n) = n+\frac{1}{2}$.
In the following we present several numerical results on the
computation of $s$th order derivative of $f(x)$. For comparison, we
perform the computations with two different approaches when compute
the Chebyshev coefficients of $f(x)$:
\begin{enumerate}
\item We compute each $a_k$ by using its optimal radius
$\rho^{*}(k)$;

\item We compute all $a_k$ by choosing the same radius
$\rho = 1$ (we use the sample points of $f(x)$ on the interval
$[-1,1]$);

\end{enumerate}
In our computations, each $a_k$ is evaluated by using the
trapezoidal rule with $m = 100$. Numerical results are presented in
Figure \ref{fig:Cheb root derivatives one}. As can be seen, our
approach is advantageous when we compute the roots of derivatives.
In Figure \ref{fig:Cheb root derivatives two} we illustrate the
results for the roots of higher order derivatives.

\end{example}

\begin{figure}[h]
\centering
\begin{overpic}
[width=7cm]{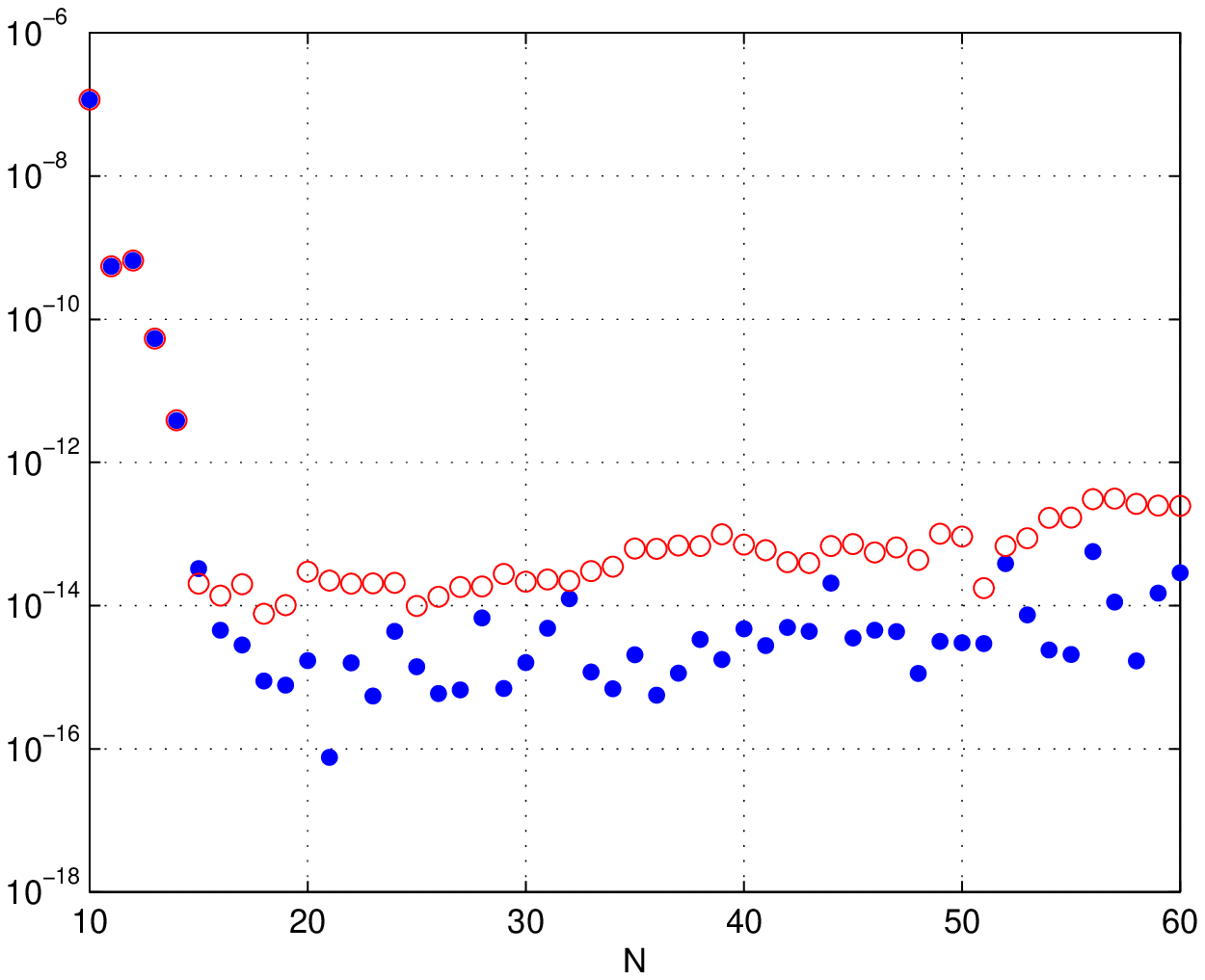}
\end{overpic}
\quad
\begin{overpic}
[width=7cm]{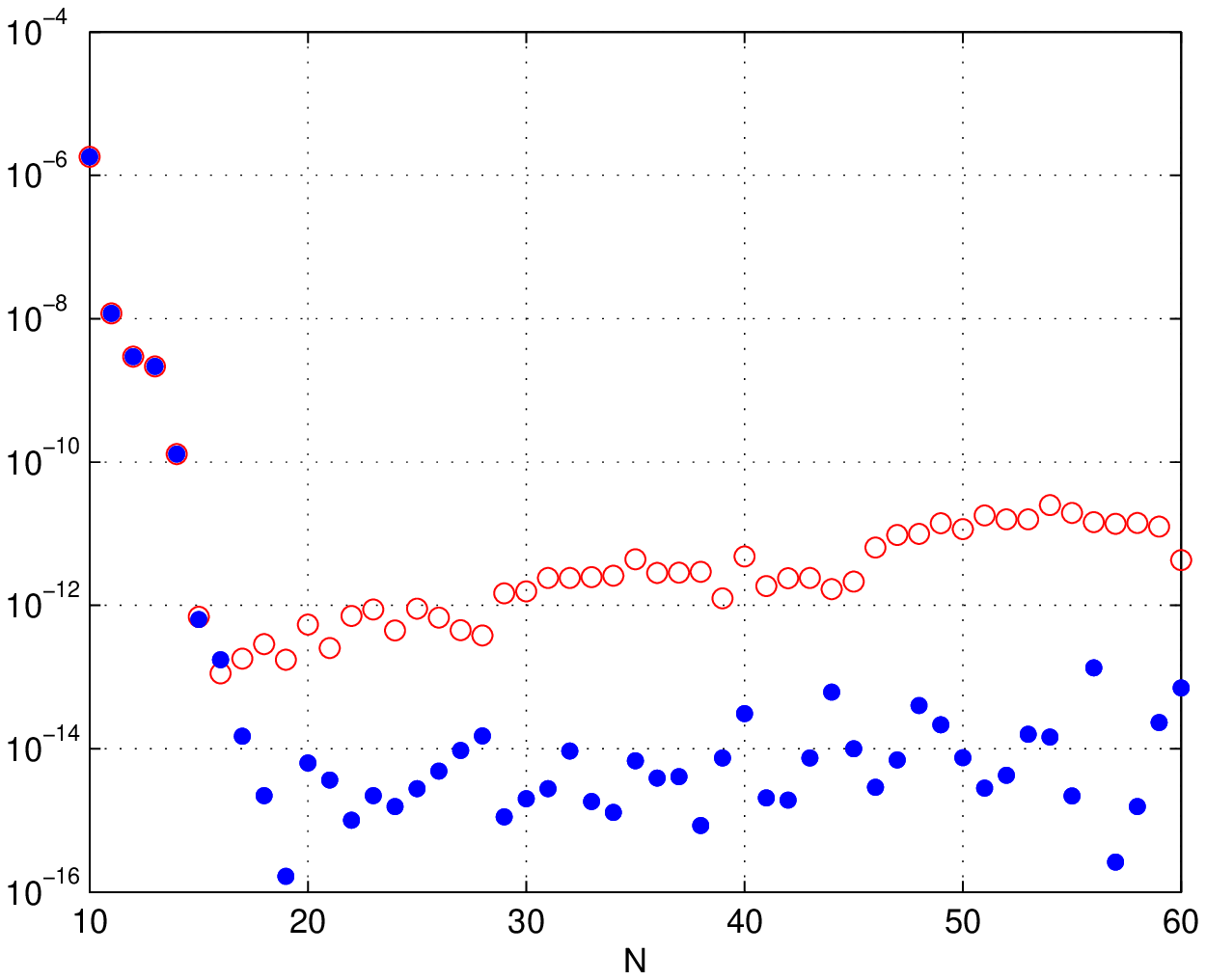}
\end{overpic}
\caption{Errors of the root of the $s$th order derivative of the
truncated Chebyshev expansion $f_N^{C}(x)$ for $s=1$ (left) and
$s=2$ (right). The dots denote the results of the strategy that each
$a_k$ is evaluated by the $\rho^{*}(k)$ and the circles denote the
strategy that all $\{a_k\}_{k=0}^{N}$ are computed by setting
$\rho=1$. } \label{fig:Cheb root derivatives one}
\end{figure}

\begin{figure}[h]
\centering
\begin{overpic}
[width=7cm]{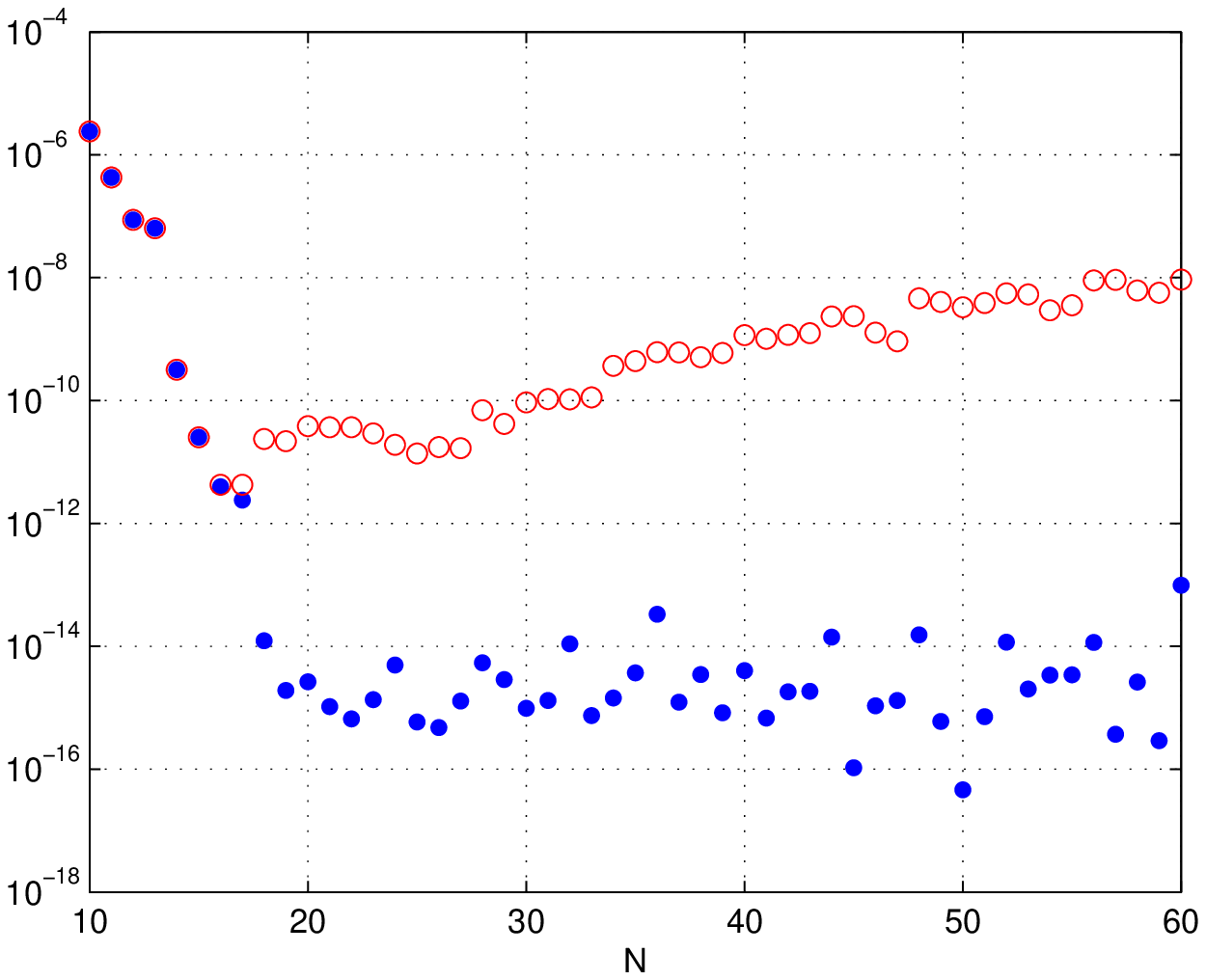}
\end{overpic}
\quad
\begin{overpic}
[width=7cm]{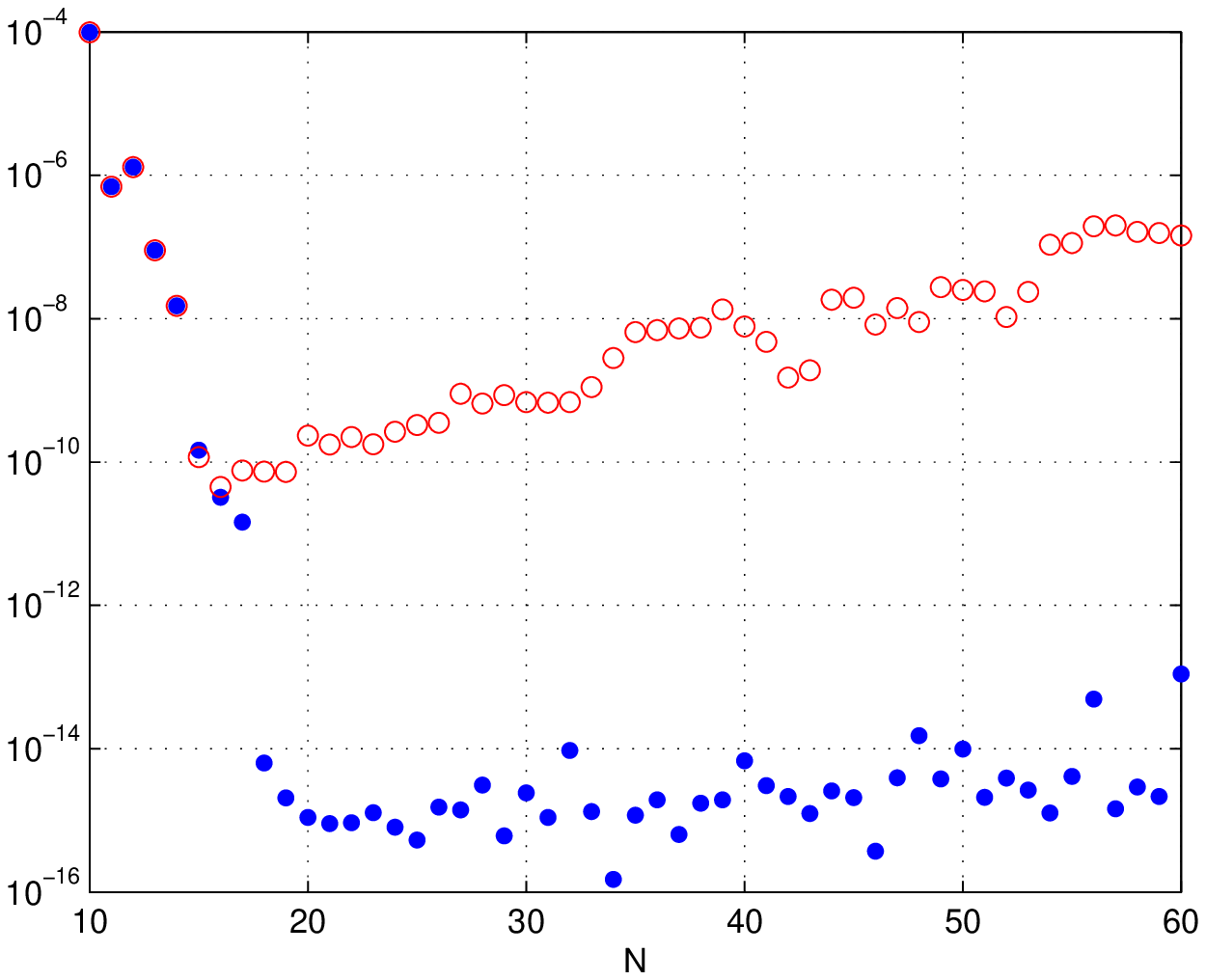}
\end{overpic}
\caption{Errors of the root of the $s$th order derivative of the
truncated Chebyshev expansion $f_N^{C}(x)$ for $s=4$ (left) and
$s=5$ (right). The dots denote the results of the strategy that each
$a_k$ is evaluated by the $\rho^{*}(k)$ and the circles denote the
strategy that all $\{a_k\}_{k=0}^{N}$ are computed by setting
$\rho=1$. } \label{fig:Cheb root derivatives two}
\end{figure}

\section{Conclusion}\label{sec:conclusion}
In this paper, we have discussed the computation of Chebyshev
expansion coefficients of analytic functions. Two strategies have
been proposed based on the computational accuracy and efficiency of
the Chebyshev expansion coefficients. The first strategy is that we
compute all Chebyshev coefficients using the same contour and this
process can be performed efficiently via the FFT. However, this
strategy may not be stable with respect to relative errors.
Alternatively, we propose the second strategy by extending the idea
of Bornemann's analysis for the Taylor coefficients to the Chebyshev
coefficients. We show that an optimal contour exists for each
Chebyshev expansion coefficient. Computing each Chebyshev expansion
coefficient with the optimal radius guarantees the relative error to
be small. We further applied the second strategy to compute
derivatives of analytic functions by differentiating the Chebyshev
expansion. Numerical experiments show that this strategy provides
very accurate approximation even for very high order derivatives.
Finally, we apply this strategy to compute the roots of derivatives
of analytic functions.

The main focus of this paper has been to investigate the benefits of computing Chebyshev coefficients in the complex plane. Several questions remain, and are topic of further research:
\begin{itemize}
\item Can the optimal radius be deduced automatically and numerically?
\item What is an appropriate number of quadrature points to use along the contour in the complex plane, for a given radius and a given coefficient $a_n$?
\end{itemize}

\section*{Acknowledgments}

The first author is supported by the National Natural Science
Foundation of China under grant 11301200. This research was started
while the first author was a Post-Doctoral Research Fellow at the
University of Leuven. The second author is supported by FWO Flanders
projects G.0617.10, G.0641.11 and G.A004.14.

\bibliographystyle{abbrv}
\bibliography{jacobi}

\end{document}